\documentclass[10pt]{article}

\usepackage{latexsym,color,amsmath,amsthm,amssymb,amscd,amsfonts,arydshln,cancel,setspace}
\usepackage{bbm}

\newtheorem{theo}{Theorem}
\newtheorem{lem}[theo]{Lemma}
\newtheorem{cor}[theo]{Corollary}
\newtheorem{prop}[theo]{Proposition}
\newtheorem{defn}[theo]{Definition}

\def\R{\R}

\def\R{{\mathbb R}}

\def\qed{\hfill $\vcenter{\hrule height .3mm
\hbox {\vrule width .3mm height 2.1mm \kern 2mm \vrule width .3mm
height 2.1mm} \hrule height .3mm}$ \bigskip}

\def\lam{\lambda}

\def\to{\rightarrow}

\def\pmx{\begin{pmatrix}}
\def\emx{\end{pmatrix}}
\def\Hess{{\nabla^2}}

\def\det{{\rm det}}

\def\R{\mathbb R}

\allowdisplaybreaks

\usepackage{titlesec}
\titleformat{\section}{\normalfont\scshape\filcenter}{\thesection}{1em}{}
\titleformat{\subsection}[runin]{\normalfont\bfseries}{\thesubsection.}{1em}{}
\titleformat{\subsubsection}[runin]{\normalfont\bfseries}{\thesubsubsection.}{1em}{}

\begin{document}

\title{
Extremal affine surface areas in a functional setting
\footnote{Keywords: s-concave functions, (functional) affine surface areas. 2010 Mathematics Subject Classification:  52A20,  60B}}

\author{Stephanie Egler\thanks{Partially supported by  NSF grant DMS-2103482} and Elisabeth M.  Werner\thanks{Partially supported by  NSF grant DMS-2103482}}

\date{}

\maketitle
\begin{abstract}
We introduce extremal affine surface areas in a functional setting.
We show their main properties. Among them are linear invariance, isoperimetric inequalities and  monotonicity properties.  
We establish a new duality formula, which shows that 
the maximal (resp. minimal) inner affine surface area of an $s$-concave function on $\mathbb{R}^n$ equals the maximal  (resp. minimal) outer affine surface area of its Legendre polar.
We estimate the ``size" of these quantities:  up to a constant depending on $n$ and $s$ only,  the extremal affine surface areas  are proportional to a power of the integral of $f$.
This extends results obtained in the setting of convex bodies. We recover and improve those as a corollary to our results. 
\end{abstract}

\begin{spacing}{0}
{\footnotesize \tableofcontents}
\end{spacing}

\section{Introduction}
In the study of geometric inequalities  it is often important to find the proper positioning for the quantities under consideration.
Among such positions for convex bodies (compact, convex subsets of $\mathbb{R}^n$ with non-empty interior)  are  
the John-  and  the L\"owner-position.  The John-position concerns the ellipsoid of maximal volume contained in $K$, now 
called  the John ellipsoid of \(K\) after F. John who showed that  such an ellipsoid exists and is unique.
Dual to the John ellipsoid is the L\"owner ellipsoid, the ellipsoid of minimal volume containing $K$. 
These ellipsoids 
play  fundamental roles in asymptotic convex geometry.  They are  related to the isotropic position,  
to the study of volume concentration, volume ratio, reverse isoperimetric inequalities, Banach-Mazur distance of normed spaces, and many more.
We refer to e.g., the books \cite{ArtsteinGiannopoulosMilmanBook, BGVV14} for the details and more information.
\par
\noindent
In affine differential geometry, another extremal quantity  has been put forward:  Instead of asking for maximal resp. minimal volume ellipsoids, 
one asks for convex bodies $K_0$ contained in a given convex body $K$ with {\em maximal affine surface area}. 
The affine surface area  $as(K)$ of a convex body $K$ was introduced  by W. Blaschke \cite{Blaschke} in dimension $2$ and $3$ and for smooth enough bodies
and then extended to arbitrary convex bodies and  all dimensions by \cite{Leichtweiss1986, Lutwak96, SchuettWerner1990}. It is defined  as follows, 
\begin{equation*} \label{def:1affine}
as(K)=\int_{\partial K}\kappa(x)^{\frac{1}{n+1}} d\mu(x), 
\end{equation*}
where  $\kappa(x)$ is the (generalized) Gauss-Kronecker curvature at $x\in \partial K$, the boundary of $K$  and $\mu$ is the usual surface area measure on $\partial K$.
\par
\noindent
The question of finding the maximal affine surface area body  goes back to at least,  V. I. Arnold \cite{Arnold} and A. M. Vershyk
\cite{Vershik}. By compactness and upper semi continuity of affine surface area,  the supremum is attained  for some convex body $K_0 \subset K$.
However, only in dimension $2$ has the maximal affine surface area body $K_0$ been  determined exactly  by I. B\'ar\'any \cite{Barany}. Moreover, I. B\'ar\'any showed in \cite{Barany} that in the plane the extremal body $K_0$  is unique and that 
$K_0$ is the limit shape of lattice polygons contained in $K$.  
In  dimensions $3$ and higher, the maximal affine surface area body $K_0$ has not been determined.
\par
\noindent
Relatively tight estimates have been given in \cite{GiladiHuangSchuettWerner}  for the size of the maximal affine surface area, namely
it was shown  there that with  the isotropic constant $L_K$ (see subsection \ref{SubS:sinN} and (\ref{LK})) and an absolute constant $c$, 
\begin{equation}\label{Ausdruck1} 
\left(\frac{c}{L_K} \right)^\frac{2n}{n+1} \frac{1}{ n^{5/6}} \,  n \,  \text{vol}_n(B^n_2)^ \frac{2}{n+1} \leq \frac{\sup_{L\subseteq K} as (L)}{\text{vol}_n(K)^\frac{n-1}{n+1} }  \leq  n \,  \text{vol}_n(B^n_2)^ \frac{2}{n+1},
\end{equation}
where $\text{vol}_n(K)$ is the volume of $K$. Thus maximal  affine surface area 
is proportional to 
a power of $\text{vol}_n(K)$,  up to a constant  depending on $n$ only.
\par
\noindent
In \cite{GiladiHuangSchuettWerner},  the authors also show such estimates for  the  {\em extremal $L_p$-affine surface areas}. 
The $L_p$-affine surface areas were  introduced for $p>1$ by 
E. Lutwak  in his ground breaking paper  \cite{Lutwak96},  thus  initiating the far reaching $L_p$-Brunn-Minkowski theory.
Extensions to all real $p$ were given in  \cite{SchuettWerner2004}, (see also \cite{Hug, MeyerWerner2000}).  
For real  $p \neq -n$, the {\em $L_p$-affine surface areas} are 
\begin{equation} \label{def:paffine}
as_{p}(K)=\int_{\partial K}\frac{\kappa(x)^{\frac{p}{n+p}}}
{\langle x,N(x)\rangle ^{\frac{n(p-1)}{n+p}}} d\mu(x), 
\end{equation}
where $N(x)$ is the outer unit normal vector at $x$ to $\partial K$ and  $\langle
\cdot, \cdot \rangle$ is the standard inner product on $\R^n$, which induces the Euclidean norm $| \cdot|$.
The case $p=1$ is the above classical affine surface area. 
\par
\noindent
Due to its remarkable properties, $L_p$-affine surface area  is  important in many areas of mathematics and applications.
 We only quote characterizations 
of $L_p$-affine surface areas by  M. Ludwig and M. Reitzner \cite{LudwigReitzner2010}, 
the $L_p$-affine isoperimetric inequalities, proved by 
E. Lutwak \cite{Lutwak96} for $p>1$ and for all other $p$  in \cite{WernerYe2008}. The classical case $p=1$ goes back to W. Blaschke \cite{Blaschke}. 
These inequalities  are related to various other inequalities, e.g., 
E. Lutwak, D. Yang and G. Zhang \cite{Lutwak2000, Lutwak2002}.
In particular, the affine isoperimetric inequality implies  the Blaschke-Santal\'o inequality and it proved to be
the key ingredient in the solution of many problems, see e.g. the books by R. Gardner \cite{GardnerBook} and R. Schneider \cite{SchneiderBook} and also
\cite{Ludwig2010, LutwakOliker1995, SchusterWannerer2012, TW2, WernerYe2008}.
Extensions of the notion have been given to an Orlicz setting, e.g.,  \cite{GardnerHugWeil2014, HuangLutwakYangZhang, Ludwig2010,  Ye2015},  to a functional setting  \cite{CFGLSW, CaglarWerner2014} and to the spherical and 
hyperbolic setting \cite{BesauWerner2015, BesauWerner2016}.
\newline
Applications of affine surface areas have been manifold. For instance,
affine surface area appears in best and random approximation of convex bodies by polytopes, see, e.g.,  K. B\"or\"oczky
 \cite{Boeroetzky2000, Boeroetzky2000a}, P. Gruber \cite{Gruber1983, GruberHandbook}, M. Ludwig \cite{Ludwig1999},  M. Reitzner \cite{Reitzner2002, ReitznerSurvey} and also
 \cite{GroteWerner, GroteThaeleWerner, HoehnerSchuettWerner,  Schuett1991, SchuettWerner2003} 
and has connections to, e.g., concentration of volume,  \cite{FleuryGuedonPaouris, Ludwig2010, Lutwak2002}, differential equations \cite{BoeroetzkyLutwakYangZhang, HaberlSchuster2009, HuangLutwakYangZhang, TW2, TrudingerWang2008, Zhao2016}, and 
information theory, e.g.,  \cite{ArtKlarSchuWer, CaglarWerner2014, LutwakYangZhang2002, LutwakYangZhang2004, PaourisWerner2012, Werner2012}.
\par
\noindent
Much research has been devoted to extend notions from convex geometry to a functional setting. We only quote \cite{ArtKlarMil, ArtKlarSchuWer, ColesantiLudwigMussnig2017, ColesantiLudwigMussnig, FradeliziMeyer2007, FGSZ, KolesnikovWerner, LiSchuettWerner, Rotem2020}.
A natural analog to convex bodies  are $s$-concave functions. For instance, for such functions,  
John- and L\"owner ellipsoids  have been defined  \cite{Alonso-Gutiérrez2017, ivanov2020functional, ivanov2023functional,  LiSchuettWerner2019} and  
$L_p$ affine surface areas have been extended to a functional setting as well \cite{CFGLSW}.
\par
\noindent
In this paper, we introduce 
extremal affine surface areas for $s$-concave functions. We define the {\em inner and outer maximal affine surface areas} of an $s$-concave function $f$ by
\begin{equation*}
IS^{(s)}_\lam(f) =  \sup_{h \leq f} as^{(s)}_\lam(h),           \ \   \   OS^{(s)} _\lam(f)= \sup_{h \geq f}as^{(s)}_\lam(h).
\end{equation*}
Here,  for $\lam \in \mathbb{R}$, $as^{(s)}_\lam(h)$ is the $\lam$-affine surface area of $h$, see Definition \ref{def:s},  and  the suprema are taken over all sufficiently smooth $s$-concave functions $h$.  
For the precise definition we  refer to  Definition \ref{ex-asa}, where we also introduce the {\em inner and outer minimal $\lam$-affine surface areas}.
\par
\noindent
Our definitions are not the exact equivalent of the  ones for convex bodies. One consequence of this modified definition is  a new duality formula, that had not been observed in the convex body case,   between the inner maximal affine surface 
area of an $s$-concave function $f$ and the outer maximal affine  surface area of its Legendre dual function $f_{(s)}^\circ$, 
\begin{equation}\label{duality1}
IS^{(s)}_\lam(f) = OS^{(s)}_{1-\lam}(f_{(s)}^\circ),
\end{equation}
Such a duality formula also holds  for the inner and outer minimal $\lam$-affine surface areas, see Theorem \ref{basic}.
In Theorem \ref{basic} we also prove invariance properties, isoperimetric inequalities and  monotonicity properties in $s$ and in $\lam$ for the extremal affine surface areas. 
\par
\noindent
As in the case of convex bodies, we give relatively tight estimates for the size of the extremal functional affine surface areas. For instance, we show in Theorem \ref{prelprop},
that  with  the isotropic constant $L_f$ of $f$,  see (\ref{Lphi}),  and an absolute constant $c$, 
$$
\max\left\{\frac{c^{n_s }}{n_s^{5/6-\lam} L_f^{2n \lam}}, \frac{1}{n_s^{n_s(1-2\lam)} } \right\} \left( \int_{\mathbb{R}^n} g_e^{(s)}\right)^{2\lam}  \leq \frac{IS^{(s)}_\lam(f)}{\left( \int_{\mathbb{R}^n} f \right)^{1-2\lam}}  \leq \left( \int_{\mathbb{R}^n} g_e^{(s)}\right)^{2\lam}, 
$$
where $n_s=n+\frac{1}{s}$ and $g_e^{(s)}(x) = (1-s |x|^2)_+^\frac{1}{2s}$ is the analogue of the Euclidean ball. 
Since $$\int_{\mathbb{R}^n} g_e^{(s)} =  \frac{\pi^{\frac{n}{2}}}{s^\frac{n}{2}}  \, 
\frac{\Gamma\left(\frac{1}{2s} +1\right)}{\Gamma\left(\frac{n+\frac{1}{s}} {2} +1\right)},$$ which is asymptotically equivalent to $\frac{(2 \pi e)^\frac{n}{2}}{n_s^\frac{n_s}{2}}$, these estimates show that for fixed $\lam$ the maximal affine surface area is proportional to 
a power of ``volume" of $f$,  up to a constant  depending on $n$ and $s$ only. They   are  the analogue to (\ref{Ausdruck1}), but with a slightly better dependence, $n_s^{5/6 -\lam}$ instead of $n^{5/6}$. Subsequently we could improve the exponent 
there as well. The fact that we have two lower bounds is a consequence of the above 
duality formula (\ref{duality1}). 
\par
\noindent
Corresponding estimates for the other extremal affine surface areas are also proved in Theorem \ref{prelprop} and Proposition \ref{prelprop1}.
\par
\noindent
Among the tools used for the proofs are the (functional) Blaschke Santal\'o inequality,  the L\"owner ellipsoid, and  a thin shell estimate by Y.T.  Lee and S. S. Vempala \cite{LeeVempala17}, O. Gu\'edon and E. Milman \cite{GuedonMilman}, see also G. Paouris \cite{Paouris}, on concentration of volume. 
\par
\noindent
As a consequence of Theorem \ref{prelprop} and Proposition \ref{prelprop1} we recover the results obtained in \cite{GiladiHuangSchuettWerner} for convex bodies 
(with an additional mild assumption). For instance for $p=1$, we get the following  slightly stronger version of (\ref{Ausdruck1}), 
\begin{equation*}\label{Ausdruck1} 
\max \left\{\left(\frac{c}{L_K} \right)^\frac{2n}{n+1} \frac{1}{ n^{5/6-1/(n+1)}}, \frac{1}{n^{n\frac{n-1}{n+1}}} \right\}\,  n \,  \text{vol}_n(B^n_2)^ \frac{2}{n+1} \leq \frac{\sup_{L\subseteq K} as (L)}{\text{vol}_n(K)^\frac{n-1}{n+1} }  \leq  n \,  \text{vol}_n(B^n_2)^ \frac{2}{n+1}.
\end{equation*}

\par
\noindent
\subsection{Notation.}
We work in $n$-dimensional Euclidean space $\mathbb{R}^n$ equipped
with the standard inner product and Euclidean norm $| \cdot |$. We denote by $B^n_2(a,r)$ the $n$-dimensional Euclidean ball with radius $r$ centered at $a$. We write in short $B^n_2=B^n_2(0,1)$. $S^{n-1}$ is the Euclidean sphere.
 The interior  and
boundary of a set $A \subset  \mathbb{R}^n$ are denoted by $\text{int}(A)$ and $\partial A$,  respectively.
The volume  of a Lebesgue measurable set $A$ in $\mathbb{R}^n$ is $\text{vol}_n(K)$. 
We write in short, in particular in long formulas,   $|K|$  instead of $\text{vol}_n(K)$.
\par
\noindent
Finally, $c$, $C$, $C_0$, etc.   are absolute constants that may change from line to line.

\section{Background}

\subsection{ $s$-concave functions.}

Let $s > 0$, $s \in \mathbb{R}^n$ and let $f: \mathbb{R}^n \rightarrow \mathbb{R}_+$.  Following Borell \cite{Borell1975}, 
we say that  $f$ is $s$-concave if for every $\lambda \in [0,1]$ and all  $x$ and $y$ such that $f(x) >0$ and $f(y) > 0$,
\[
f( (1-\lam)x + \lam y) \ge \left( (1-\lam) f(x)^s + \lam f(y)^s \right)^{1/s}.
\]
Since $s>0$, one may equivalently assume that $f^s$ is concave on its support,  the  convex set $ S_f=\{x: f(x) >0\}$.
\par
\noindent
We always consider in this paper  $s$-concave functions that are integrable and non-degenerate,  
i.e., the interior of the support of $f$ is non-empty, $\operatorname{int} (S_f) \neq \emptyset$. 
This means that $0 < \int _{\mathbb{R}^n} f\, dx < \infty$.
Note that $f$ is continuous on $\operatorname{int} (S_f)$ and  we may assume that $f$ is upper semi continuous.
We may also assume that $0\in\operatorname{int} (S_f)$, see Section \ref{def-ex}.
\par
\noindent
We will denote the class of these  $s$-concave functions by ${\mathit C}_s(\R^n)$.
Observe that if 
\begin{equation}\label{contain}
s_1 \leq s_2,  \hskip 3mm \text{then }
\hskip 3mm {\mathit C}_{s_2}(\R^n) \subseteq {\mathit C}_{s_1}(\R^n).
\end{equation} 
The $s$-concave function $g^{(s)}_{e,r}: \mathbb{R}^n \rightarrow \mathbb{R}_+$ that plays the role of the Euclidean ball with radius $r>0$ in ${\mathit C}_{s}(\R^n)$ is 
\begin{equation}\label{r-s-ball}
g^{(s)}_{e,r}(x) = (r^2-s|x|^2)_+^\frac{1}{2s},
\end{equation}
where, for $a \in \mathbb{R}_+$, $a_+= \max \{a,0\}$.
When $r=1$, we write in short
$$g^{(s)}_e(x) = (1-s |x|^2)_+^\frac{1}{2s}.
$$
\vskip 3mm
\noindent
We refer to the books  \cite{GardnerBook, Rockafellar, SchneiderBook} for background on convex functions and convex bodies.
\vskip 4mm
\noindent
\subsection {An associated body.}
Sometimes it is convenient to pass from the functional setting to the convex body setting. This can be achieved  in the following way. For $s >0$ such that $\frac{1}{s} \in \mathbb{N}$, we  associate with an $s$-concave function  $f$ a  convex body $K_{s}(f)$ \cite{ArtKlarMil} (see also \cite{ArtKlarSchuWer})
in $\R^{n}\times \R^\frac{1}{s}$, 
\begin{equation}\label{def.Ksf}
K_s(f)=\big\{(x,y) \in \R^n \times \R^\frac{1}{s}: x/\sqrt{s} \in
\overline{S_f}, |y| \leq
f^s (x/\sqrt{s})\big\}.\end{equation}
Note that $K_s(g^{(s)}_{e,r})= B_2^{n+\frac{1}{s}}(0, r)= r B_2^{n+\frac{1}{s}}$.
\par
\noindent
Then,  
\begin{equation}\label{vol-Ks}
 \text{vol}_{n+\frac{1}{s}} \left(K_s(f)\right) = s^\frac{n}{2}  \text{vol}_{\frac{1}{s}} \left(B_2^\frac{1}{s}\right)\, \int_{\mathbb{R}^n} f\, dx
\end{equation}
and 
\begin{equation}\label{vol-rKs-euclid}
 \text{vol}_{n+\frac{1}{s}} \left(K_s(g^{(s)}_{e,r})\right) = s^\frac{n}{2} \text{vol}_{\frac{1}{s}} \left(B_2^\frac{1}{s}\right)\, \int_{\mathbb{R}^n} g^{(s)}_{e,r} \, dx = 
  \text{vol}_{n+\frac{1}{s}} \left(r\, B_2^{n+\frac{1}{s}}\right).
\end{equation}
In particular
\begin{equation}\label{vol-Ks-euclid}
 \text{vol}_{n+\frac{1}{s}} \left(K_s(g^{(s)}_e)\right) = s^\frac{n}{2} \text{vol}_{\frac{1}{s}} \left(B_2^\frac{1}{s}\right)\, \int_{\mathbb{R}^n} g^{(s)}_e \, dx =  \text{vol}_{n+\frac{1}{s}} \left(B_2^{n+\frac{1}{s}}\right).
\end{equation}
\vskip 3mm
\noindent
\subsection{ $\lam$-affine surface area for  $s$-concave functions.}
When the  function $f\in C_s(\mathbb{R}^n)$ is sufficiently smooth, the gradient of $f$, denoted  by $\nabla f$, and the Hessian of $f$, denoted  by
$\nabla^2 f$ exist everywhere in the interior of $S_f$.
In the general case, when $f$ is not smooth enough, 
the gradient of $f$ exists almost everywhere by Rademacher's theorem (see, e.g., \cite{Rademacher}),   and a theorem of Alexandrov \cite{Alexandroff} and Busemann and Feller \cite{Buse-Feller} guarantees the existence of the Hessian
almost everywhere in the interior of $S_f$. Let $X_f$  be the set of points of $\operatorname{int} (S_f)$ at which its Hessian $\nabla^2 f$ in the sense of Alexandrov is defined and  invertible.
\vskip 3mm
\noindent
For functions $f \in  {\mathit C}_s(\R^n)$, the $\lambda$-affine surface areas 
have been introduced 
in \cite{CFGLSW} 
as follows. 
\par
\noindent
\begin{defn}  \cite{CFGLSW}
\label{def:s}
Let $s > 0$ and  let $f \in C_s(\mathbb{R}^n)$.  
For any $\lam \in \R$, let 
\begin{equation}\label{def-1}
as_{\lam}^{(s)}(f) = \frac{1}{s^{n \lam} (1+ns)} \, \int_{X_f} \frac{f^s(x)^{ \left(\frac{1}{s}-1\right)(1-\lam)}
 \left(\det (-\Hess (f^s (x) ))\right)^\lam}
 {\left(f^s(x) -\langle x , \nabla (f^s)(x) \rangle\right)^{\lam\left(n+\frac{1}{s}+1\right) - 1 }} \  dx.
\end{equation}
\end{defn}
\par
\noindent
By the  remarks at the beginning of this section, the $\lam$-affine surface areas are well defined. 
The normalizing factors in the expression are chosen so that $as_{\lam}^{(s)}(g_e) =\int g_e dx$, see Lemma \ref{lemma-phie} below.
Note also that since $\det (\nabla^2 (f^s)(x))=0$ outside $X_f$, the integral may be taken on $\operatorname{int} (S_f)$ for $\lambda>0$. The expression (\ref{def-1}) can be rewritten as
\begin{equation*}
as_{\lam}^{(s)}(f) = \frac{1}{1+ns}\,  \int_{X_f} \frac{f^{2 (1-\lam) -s}
 \left(\det (- \left(\frac{\Hess f }{f} +(s-1) \frac{\nabla f \otimes \nabla f}{f^2}\right))\right)^\lam}
 {\left(1 -s \langle x , \frac{\nabla  f}{f} \rangle \right)^{\lam\left(n+\frac{1}{s}+1\right) - 1 }} \  dx.
\end{equation*}
\vskip 2mm
\noindent
There is yet another way how to express the  $\lam$-affine surface area of an $s$-concave function.  In fact, it was in this way that it was originally introduced in \cite{CFGLSW}.
We define a  convex function $\psi$ on $S_f$  by
\begin{equation}\label{def:psi}
\psi(x)= \frac{1-f^s(x)}{s},   \   \ x\, \in \,  S_f,
 \end{equation}
extend it by continuity to the closure of $S_f$ and by $+\infty$ outside the closure of $S_f$. 
Then
\begin{equation}\label{def-2}
as_{\lam}^{(s)}(f) = \frac{1}{1+ns} \,  \int_{X_\psi} \frac{\left(1-s \psi(x) \right)^{\left(\frac{1}{s}-1\right)(1-\lam)}
 \left(\det   \Hess  \psi (x) \right)^\lam}
 {\left(1+s(\langle x , \nabla \psi(x) \rangle -\psi(x))\right)^{\lam\left(n+\frac{1}{s}+1\right) - 1 }} \  dx, 
\end{equation}
where $X_\psi$ is the set of points of $\Omega_\psi$ at which its Hessian $\nabla^2\psi$ in the sense of Alexandrov is defined and  invertible
and $\Omega_\psi$ is the interior of the convex domain of $\psi$, that is 
\[
\Omega_\psi= \mathrm{int} \, (\{x \in \R^n, \psi(x) < +\infty\} ) =  \mathrm{int} \, (S_f). \]
It is easy to see that (\ref{def-1}) and (\ref{def-2}) are equivalent.
It was noted in \cite{CFGLSW} that 
\begin{equation}\label{def:asa-0}
as_{0}^{(s)} (f) = \int_{S_{f}} f \, dx = \int_{\mathbb{R}^n} f(x) dx.
\end{equation}
\vskip 2mm
\noindent
The next identities explain why  we call the quantities $as_{\lambda}^{(s)}$ functional affine surface areas.
\par
\noindent
For $\frac{1}{s} \in \mathbb{N}$, it follows from Proposition 5 in  \cite{CaglarWerner2014} that 
the $\lam$-affine surface area of a sufficiently smooth $f \in {\mathit C}_{s}(\R^n)$ is a multiple of the  $p$-affine surface area of $K_{s}(f)$ with $p=\left(n+\frac{1}{s}\right)\ \frac{\lam}{1-\lam}$, 
$p \neq - \left(n+\frac{1}{s}\right)$, 
\begin{equation}\label{motiv1}
as_{\lambda}^{(s)} (f)  = \frac{as_{p} \left(K_s(f) \right)}{ (n +\frac{1}{s})\,  s^\frac{n}{2} \, \text{vol}_{\frac{1}{s}} \left(B_2^{\frac{1}{s}}\right) }.
\end{equation}
Note that for $\lam=0$ we get 
\begin{equation}\label{lam=0}
\int_{S_{f}} f \, dx = as_{0} ^{(s)}(f)  =  \frac{\text{vol}_{\frac{1}{s}} \left(K_s(f) \right)}{  s^\frac{n}{2} \, \text{vol}_{\frac{1}{s}} \left(B_2^{\frac{1}{s}}\right) }
\end{equation}
and for 
$\lam = 1$
\begin{equation}\label{lam=1}
 as_{1} ^{(s)}(f)  = \frac{as_{\infty} \left(K_s(f) \right)}{ (n +\frac{1}{s})\,  s^\frac{n}{2} \, \text{vol}_{\frac{1}{s}} \left(B_2^{\frac{1}{s}}\right) }.
\end{equation}
\vskip 3mm
\noindent
Before we  recall more  facts from \cite{CFGLSW},  we will need a notion of  dual of an $s$-concave function. For the  construction of such a duality  we follow \cite{ArtKlarMil, Meyer91}. See also \cite{CFGLSW}.
We  define the $(s)$-Legendre dual of $f$ 
as 
\begin{equation}\label{Leg}
f_{(s)}^\circ(y)  = \inf_{x \, \in  \, S_f}  \frac{(1 - s \langle x, y \rangle)_+^{1/s}}{f(x)},
\end{equation}
with $S_{f_{(s)}^\circ} = \frac{1}{s} \left(S_f\right)^\circ$.
For   $f  \in {\mathit C}_s(\mathbb{R}^n)$, 
 $f_{(s)}^\circ \in {\mathit C}_s(\mathbb{R}^n)$. It is not difficult to see that as for the Legendre transform, 
$(f_{(s)}^\circ)_{(s)}^\circ = f$. 
Note also that $(g^{(s)}_e)_{(s)}^\circ = g^{(s)}_e$, a further indication  that $g^{(s)}_e$ plays the role of the Euclidean ball.
It was shown in \cite{CFGLSW} that 
\begin{equation}\label{def:asa-1}
as_1^{(s)} (f) = \int_{S_{f_{(s)}^{\circ}}} f_{(s)}^{\circ} (x) dx.
\end{equation}
\vskip 2mm
\noindent
We denote by  ${\mathit C}_s^+(\R^n)$  the set of 
functions $f \in {\mathit C}_s(\R^n)$ 
that 
are  twice continuously differentiable on $S_f$,  have non-zero  Hessian  on $S_f$ and  $\lim_{x \rightarrow \partial (S_f) } f^s(x)=0$.
Observe that for $f \in C_s^+(\mathbb{R}^n)$ 
we have that $\left(f_{(s)}^\circ\right)_{(s)}^\circ=f$, hence $f \to f_{(s)}^\circ$ is a bijection on $C_s^+(\mathbb{R}^n)$.
\vskip 3mm
\noindent
The following results were shown in \cite{CFGLSW}.
\vskip 2mm
\begin{prop}  \cite{CFGLSW} \label{Ungleichung1} 
\label{properties}
Let  $s > 0$ and $f \in C_s(\mathbb{R}^n)$. 
\par
\noindent
(i) For all linear maps $T: \mathbb{R}^n \to \mathbb{R}^n$ and for all $\lam \in \mathbb{R}$, 
\begin{equation*}
as_\lam^{(s)}(f \circ T) = |\det T| ^{2 \lam-1}\, as_\lam ^{(s)}(f).
\end{equation*}
\par
Moreover, for all $\alpha \in \mathbb{R}$, $\alpha > 0$, $as_\lam^{(s)}( \alpha \, f ) =  \alpha^{1-2 \lam}\, as_\lam ^{(s)}(f)$.
 \vskip 2mm
 \noindent
(ii) Let $f \in {\mathit C}_s^+(\R^n)$.  Then 
\begin{eqnarray} \label{ Ungleichung1}  \, 
 \forall \lambda \in [0,1], 
&\quad as_{\lambda}^{(s)} (f)  \leq  \left(\int_{\R^n}f \  dx \right)^{1-\lambda}
 \left( \int_{\R^n}  f^\circ_{(s)}  \  dx \right)^\lambda  ; \nonumber\\
\forall \lambda \notin [0,1], 
& \quad as_{\lambda}^{(s)} (f)  \geq  \left(\int_{\R^n}f \  dx \right)^{1-\lambda}
 \left( \int_{\R^n}  f^\circ_{(s)}  \  dx \right) ^\lambda.
\end{eqnarray}
 Equality holds in the first inequality if $\lam=0$ or $\lam=1$.
\vskip 2mm
\noindent
(iii) For  $f \in {\mathit C}_s^+(\R^n)$,  for all $\lambda  \in \mathbb{R}$, \hskip 3mm $as_{\lambda}^{(s)} (f) = as_{1-\lambda}^{(s)} (f_{(s)}^\circ) $.
\end{prop}
\vskip 3mm
\noindent
We denote by ${\mathit PL}_s(\R^n)$ the set of functions on $\mathbb{R}^n$ that consist piecewise of functions of the form $f_l(x) =(1-s \psi_l(x))^\frac{1}{s}$, where $\psi_l(x)=\alpha - \langle a, x \rangle$, with $\alpha \in \mathbb{R}$, $a \in \mathbb{R}^n$, $a \neq 0$.
\vskip 2mm
\noindent
The following lemma is easy to see.
\begin{lem} \label{lemma-phie}
Let $s \in \mathbb{R}$, $s>0$. 
Let $\lam \in \mathbb{R}$. Then
\par
(i) $as_{\lambda}^{(s)} (g^{(s)}_{e}) =  \int_{\mathbb{R}^n} g^{(s)}_e dx=  \frac{ \text{vol}_{n+\frac{1}{s}} \left(B^{n+\frac{1}{s}}_2\right)}{s^\frac{n}{2}  \text{vol}_{\frac{1}{s}} \left(B_2^{\frac{1}{s}}\right)}=\frac{\pi^{\frac{n}{2}}}{s^\frac{n}{2}}  \, 
\frac{\Gamma\left(\frac{1}{2s} +1\right)}{\Gamma\left(\frac{n+\frac{1}{s}} {2} +1\right)}$. 
\par
(ii) $as_{\lambda}^{(s)} (g^{(s)}_{e,r}) =  r^{\left(n+\frac{1}{s}\right)  (1-2 \lambda)} as_{\lambda}^{(s)} (g^{(s)}_{e})$. 
In particular, $as_{\frac{1}{2}} ^{(s)}(g^{(s)}_{e,r}) =   as_{\frac{1}{2}} ^{(s)}(g^{(s)}_{e})$.
\par
(iii) Let $f_l \in {\mathit PL}_s(\R^n)$. Then
\begin{eqnarray*}
as_{\lambda}^{(s)} (f_{l}) = \left\{
\begin{aligned}
&0 \quad &\text{ if }\  \lam > 0,\\
&\infty \quad &\text{ if }\,  \lam <0 .
\end{aligned}\right.
\end{eqnarray*}
\end{lem}
\par
\noindent
Recall that by (\ref{def:asa-0}),   $as_{0}^{(s)} (f_l) = \int_{S_{f_l}} f_l \, dx$.

\section{Extremal affine surface areas for functions} \label{Ex-asa-functions}

\subsection{Definition.} \label{def-ex}

For convex bodies $K$ in $\mathbb{R}^n$, extremal affine surface areas have been defined in \cite{GiladiHuangSchuettWerner} as follows:
For $-\infty \leq p \leq \infty$, $p \neq -n$,
\begin{equation}\label{def sup-body}
IS_p(K) =  \sup_{C \subseteq  K, \, C\,  \text{convex}}\big(as_p(C)\big),           \ \   \   OS _p(K)= \sup_{K \subseteq C, \, C\,  \text{convex}}\big(as_p(C)\big),  
\end{equation}
and 
\begin{eqnarray}\label{def inf-body}
{i s}_p(K)=  \inf_{C \subseteq K, \, C\,  \text{convex}}\big(as_p(C)\big),         \ \   \   {os}_p(K)=  \inf_{ K \subseteq C, \, C\,  \text{convex} }\big(as_p(C)\big).
\end{eqnarray}
\par
\noindent
We now define the 
corresponding quantities for $s$-concave functions.  
\par
\noindent
In the case  of convex bodies $K$, it was natural to position $K$ properly, namely such that
the center of gravity is at $0$.
We do the same here and assume throughout that the  functions $f$ considered  have center of gravity at $0$, 
\begin{equation} 
\int_{S_f} x \, f dx =0.
\end{equation}
Since $S_f$ is convex, the center of gravity is in the interior of $S_f$,  $0 \in \operatorname{int}(S_f)$.
\vskip 2mm
\noindent
\begin{defn}\label{ex-asa}
Let  $\lambda \in \mathbb{R}$ and let  $f \in {\mathit C}_s(\R^n)$.
We define the {\em inner and outer maximal affine surface areas} of $f$ by
\begin{equation}\label{def sup}
IS^{(s)}_\lam(f) =  \sup_{h \leq f, h \in {\mathit C}_s^+(\R^n)} as^{(s)}_\lam(h),           \ \   \   OS^{(s)} _\lam(f)= \sup_{h \geq f, h \in {\mathit C}_s^+(\R^n)}as^{(s)}_\lam(h),  
\end{equation}
and the {\em inner and outer minimal affine surface areas} by
\begin{eqnarray}\label{def inf}
{i s}^{(s)}_\lam(f)=  \inf_{h \leq f, h \in {\mathit C}_s^+(\R^n)}as^{(s)}_\lam(h),         \ \   \   {os}^{(s)}_\lam(f)=  \inf_{h \geq f, h \in {\mathit C}_s^+(\R^n)}as^{(s)}_\lam(h). 
\end{eqnarray}
The extrema are taken over all $s$-concave functions $h$ such that the support of $h$ is contained in the support of $f$,  $S_h \subseteq S_f$, and 
$h(x) \leq f(x)$ for all $x \in S_h$ 
or such that the support of $h$ contains the support of $f$, $S_f \subseteq S_h$,
and  $h(x) \geq f(x)$ for all $x \in S_f$.
\end{defn}
\vskip 2mm
\noindent
{\bf Remarks.}  
\par
\noindent
1. Our definition of functional extremal affine surface areas differs slightly from the one for convex bodies, given in (\ref{def sup-body}) and (\ref{def inf-body}).  As we shall see in Section  \ref{main}, this ensures that now 
$is^{(s)}_\lam$ is not identical equal to $0$ which was the case for the corresponding expression in the convex body setting. This  has a duality formula as a consequence, 
see Theorem \ref{basic} (iv).
\par
\noindent
2. We show  in Theorem \ref{ranges} that there are meaningful $\lam$-ranges for the functional extremal affine surface areas. Outside these ranges, which differ for the different 
extremal surface areas,  they  are either $0$ or infinity.
\par
\noindent
3. While for $f \in PL(\mathbb{R}^n)$, $as_{\lambda}^{(s)} (f_{l}) =0$ for $\lam >0$, respectively, $as_{\lambda}^{(s)} (f_{l}) =\infty$ for $\lam <0$, we now have
 that for the meaningful $\lam$-ranges, 
 $$
IS_\lam^{(s)}(f) >0,  \,  \lam \in\left [0, \frac{1}{2}\right], \hskip 3mm  OS_\lam^{(s)}(f) > 0, \, \lam \in \left[ \frac{1}{2},1\right], 
$$
$$
 os_\lam^{(s)}(f) < \infty, \, \lam \in (- \infty, 0], \hskip 3mm  is_\lam^{(s)}(f) >0,  \,  \lam \in [1, \infty).
$$
\par
\noindent
4. Please note that  because of the modified definition, even for  $s>0$  such that $\frac{1}{s} \in \mathbb{N}$,    
$$
IS_\lam^{(s)}(f) \leq  \frac{IS_p\left(K_s(f)\right)}{(n +\frac{1}{s})\,  s^\frac{n}{2} \, \text{vol}_{\frac{1}{s}} \left(B_2^{\frac{1}{s}}\right)}, \hskip 3mm OS_\lam^{(s)}(f) \leq \frac{OS_p\left(K_s(f)\right)}{(n +\frac{1}{s})\,  s^\frac{n}{2}  \text{vol}_{\frac{1}{s}} \left(B_2^{\frac{1}{s}}\right)}, 
$$
$$
os_\lam^{(s)}(f) \geq \frac{os_p\left(K_s(f)\right)}{(n +\frac{1}{s})\,  s^\frac{n}{2} \, \text{vol}_{\frac{1}{s}} \left(B_2^{\frac{1}{s}}\right)}, \hskip 3mm 
is_\lam^{(s)}(f)\geq \frac{is_p\left(K_s(f)\right)}{(n +\frac{1}{s})\,  s^\frac{n}{2} \, \text{vol}_{\frac{1}{s}} \left(B_2^{\frac{1}{s}}\right)},
$$
where $p=\left(n+\frac{1}{s}\right) \frac{\lam}{1-\lam}$.
\vskip 4mm
\noindent

\section{Main results} \label{main}

\subsection{Basic properties of the extremal surface areas.}

Before we state the results on the extremal surface areas, we show a monotonicity property of the $\lam$-affine surface areas which we will need later.
\vskip 2mm
\begin{prop} \label{mono-asa}
Let $f \in {\mathit C}_s(\R^n)$.  For $\lam >0$,  $\left(\frac{as_{\lambda}^{(s)} (f)}{as_{0}^{(s)} (f)}\right)^\frac{1}{\lam}$ is strictly increasing in $\lam$ and for
$\lam <0$, $\left(\frac{as_{\lambda}^{(s)} (f)}{as_{0}^{(s)} (f)}\right)^\frac{1}{\lam}$ is strictly decreasing in $\lam$.
\par
\noindent
Equality holds if $f=g_e^{(s)}$.
\end{prop}
\vskip 2mm
\begin{proof} Let $\frac{\mu}{ \lam} \geq 1$. Then, with H\"older's inequality, 
\begin {eqnarray*}
as_{\lambda}^{(s)} (f) &=&  \frac{1}{ (1+ns)} \, \int_{X_f} \frac{f^s(x)^{ \left(\frac{1}{s}-1\right)(1-\lam)}
 \left(\frac{\det (-\Hess (f^s (x) ))}{s^n}\right)^\lam}
 {\left(f^s(x) -\langle x , \nabla (f^s)(x) \rangle\right)^{\lam\left(n+\frac{1}{s}+1\right) - 1 }} \  dx \\
 &=&   \frac{1}{ (1+ns)} \, \int_{X_f} \left[\frac{f^{s \left(\frac{1}{s}-1\right)(1-\mu)} \left(\det (-\Hess (f^s ))\right)^\mu}
 {s^{n \mu} \left(f^s -\langle x , \nabla f^s \rangle\right)^{\mu\left(n+\frac{1}{s}+1\right) - 1 }} 
 \right]^\frac{\lam}{\mu} \\
  && \hskip 50mm  \times\left[ f^{s \left(\frac{1}{s}-1\right) }\left(f^s -\langle x , \nabla f^s \rangle \right)\right]^{1-\frac{\lam}{\mu}} \  dx \\
&\leq& \left[ \frac{1}{ (1+ns) } \int_{X_f}\frac{f^{s \left(\frac{1}{s}-1\right)(1-\mu)} \left(\det (-\Hess (f^s ))\right)^\mu}
 {s^n \left(f^s -\langle x , \nabla f^s \rangle\right)^{\mu\left(n+\frac{1}{s}+1\right) - 1 }} dx
  \right] ^\frac{\lam}{\mu}\\
  && \hskip 35mm  \times  \left[ \frac{1}{ (1+ns) } \int_{X_f} f^{s \left(\frac{1}{s}-1\right) }\left(f^s -\langle x , \nabla f^s \rangle \right) dx \right]^{1-\frac{\lam}{\mu}}\\
&=& \left[as_{\mu}^{(s)} (f)\right]^\frac{\lam}{\mu} \left[as_{0}^{(s)} (f) \right]^{1-\frac{\lam}{\mu}}.
\end{eqnarray*}
From this, the proposition follows. Equality for $f=g_e^{(s)}$ follows from Lemma \ref{lemma-phie} (i).
\end{proof}
\vskip 3mm
\begin{theo}\label{ranges}
Let  $f \in {\mathit C}_s(\R^n)$. 
Then 
the only meaningful $\lambda$-ranges for 
$IS_\lam^{(s)}(f)$, $OS_\lam^{(s)}(f)$, $os_\lam^{(s)}(f)$ and $is_\lam^{(s)}(f)$ 
are $[0,\frac{1}{2}]$, $[\frac{1}{2},1]$,  $(-\infty,0]$  and $[1, \infty)$ respectively.
\par
\noindent
In particular, we have that
$$
IS_0^{(s)}(f) = \int_{\mathbb{R}^n} f dx, \hskip 7mm  IS_\frac{1}{2}^{(s)}(f) = \int_{\mathbb{R}^n} g^{(s)}_e dx = IS_\frac{1}{2}^{(s)}(g^{(s)}_e), $$
$$
OS_{\frac{1}{2}}^{(s)}(f)= \int_{\mathbb{R}^n} g^{(s)}_e dx = OS_{\frac{1}{2}}^{(s)}(g^{(s)}_e), \hskip 7mm  OS_{1}^{(s)}(f)  
\int_{S_{f_{(s)}^{\circ}}} f_{(s)}^{\circ} (x) dx,
$$
$$
os_0^{(s)}(f) = \int_{\mathbb{R}^n} f dx, \hskip 7mm  is_1^{(s)}(f)=\int_{S_{f_{(s)}^{\circ}}} f_{(s)}^{\circ} (x) dx.
$$
\par
\noindent
If $s$ is in addition such that $\frac{1}{s} \in \mathbb{N}$, then for the relevant $\lam$-ranges there are functions $h \in {\mathit C}_s^+(\R^n)$ such that 
$IS_\lam^{(s)}(f) = as_{\lam} ^{(s)}(h)$, $OS_\lam^{(s)}(f)= as_{\lam} ^{(s)}(h)$, $os_\lam^{(s)}(f)= as_{\lam} ^{(s)}(h)$ and $is_\lam^{(s)}(f)=as_{\lam} ^{(s)}(h)$ .
\end{theo}
\vskip 2mm
\noindent
Recall that by meaningful $\lam$-ranges we mean that the quantities are not identically $0$ or $\infty$. 
We want to point out that, in contrast to the convex body case where $is_p(K)=0$, for all $p$ and all $K$, we now have that $is_\lam^{(s)}(f)$ now is not identically equal to $0$.
\par
\noindent
We prove this theorem, as well as  the other main results in Section \ref{proofs}. We list first properties of $g^{(s)}_{e,r}$.
\vskip 3mm
\begin{lem} \label{ge-prop}
We have for the relevant $\lam$-ranges that
$$
IS_\lam^{(s)}(g^{(s)}_{e,r})=  OS_\lam^{(s)}(g^{(s)}_{e,r})= os_\lam^{(s)}(g^{(s)}_{e,r})= is_\lam^{(s)}(g^{(s)}_{e,r})
=r^{\left(n+\frac{1}{s}\right)  (1-2 \lambda)}  \, \int_{\mathbb{R}^n} g^{(s)}_e dx. 
$$
In particular, 
$ IS_\lam^{(s)}(g^{(s)}_e)=  OS_\lam^{(s)}(g^{(s)}_e)= os_\lam^{(s)}(g^{(s)}_e)=  is_\lam^{(s)}(g^{(s)}_{e})= \int_{\mathbb{R}^n} g^{(s)}_e dx$. 
\end{lem}
\begin{proof}
For $\lam \in [0,1/2]$, we have by Lemma \ref{lemma-phie},
\begin{eqnarray*}
IS_{\lam}^{(s)}(g^{(s)}_{e,r}) &=&  \sup_{h \leq g^{(s)}_{e,r}, h \in {\mathit C}^+_{s}(\R^n) } as_{\lam} ^{(s)}(h) \geq as_{\lam} ^{(s)}(g^{(s)}_{e,r}) = r^{\left(n+\frac{1}{s}\right)  (1-2 \lambda)} as_{\lambda}^{(s)} (g^{(s)}_{e})\\
&=& r^{\left(n+\frac{1}{s}\right)  (1-2 \lambda)}  \int_{\R^n}  g^{(s)}_e\, dx.
\end{eqnarray*}
On the other hand, for  $\lam \in [0, \frac{1}{2}]$,  we have by Proposition \ref{Ungleichung1}, 
\begin{eqnarray*}
IS_{\lam}^{(s)}(g^{(s)}_{e,r}) &=&  \sup_{h \leq g^{(s)}_{e,r}, h \in {\mathit C}^+_{s}(\R^n) } as_{\lam} ^{(s)}(h) \\
&\leq& \sup_{h \leq g^{(s)}_{e,r}, h \in {\mathit C}^+_{s}(\R^n)} \left(\int_{\R^n}  h\, dx \right)^{1-\lam}
 \left( \int_{\R^n}  h_{(s)}^\circ  \  dx \right)^{ \lam}.
\end{eqnarray*}
We recall now a general form of the functional Blaschke-Santal\'o inequality  \cite{FradeliziMeyer2007,Lehec2009b}.  Let $f$ be a non-negative integrable 
function on $\R^n$. There exists $z_0\in\R^n$ such that for every measurable $\rho:\R_+\to\R_+$ and every $g : \R^n\to\R_+$ satisfying
\begin{equation}\label{hypsantalo}
f (z_0+ x ) g ( y ) \le \left(\rho(\langle x,y\rangle)\right)^2 \hskip 2mm  \text{for every} \hskip 2mm  x,y \in \R^n \hskip 2mm  \text{with} \hskip 2mm  \langle x,y\rangle >0, 
\end{equation}
we have 
\begin{equation}\label{concsantalo}
\int_{\R^n} f \, dx  \int_{\R^n} g \, dx  \leq \left(\int_{\R^n}\rho(|x|^2)dx\right)^2 . 
\end{equation}
We apply it with $f=h$, $g=h_{(s)}^\circ$ and $\rho(t) = (1-s t)_+^\frac{1}{2s}$. 
A result of  \cite{Lehec2009b} states  that the  point $z_0$ in the above  inequality \eqref{hypsantalo} can be taken equal to $0$ when 
$\int _{\R^n} x fdx=0$. As by definition (\ref{Leg}) of $h_{(s)}^\circ$,
$$
h(x)\,  h_{(s)}^\circ(y) \leq \left(1- s \langle x, y \rangle \right)^\frac{1}{s}, 
$$
we therefore get 
\begin{equation}\label{Bl-Sa}
\int_{\R^n} h \, dx  \int_{\R^n} h_{(s)}^\circ \, dx   \leq \left(\int_{\R^n}   \left(1- s |x|^2\right)_+^\frac{1}{2 s}\, dx\right)^2
\end{equation}  
and hence
\begin{eqnarray*}
IS_{\lam}^{(s)}(g^{(s)}_{e,r}) &\leq&  \sup_{h \leq g^{(s)}_{e,r}, h \in {\mathit C}^+_{s}(\R^n)}  \left(\int_{\R^n}  h \,  dx \right)^{1-2 \lam} \left(\int_{\R^n}   \left(1- s |x|^2\right)_+^\frac{1}{2 s}\, dx\right)^{2 \lam}\\
&\leq&  \left(\int_{\R^n}  g^{(s)}_{e,r} \,  dx \right)^{1-2 \lam} \left(\int_{\R^n}  g^{(s)}_e dx\right)^{2 \lam}.
\end{eqnarray*}
As $\int_{\frac{r}{\sqrt{s}} B^n_2} g^{(s)}_{e,r}\, dx =  r^{n+\frac{1}{s}} \, \int_{\frac{1}{\sqrt{s}}B^n_2} g^{(s)}_{e}\, dx $, the claim follows.
\par
\noindent
The proofs for  $OS_\lam^{(s)}$, $os_\lam^{(s)}$ and $is_\lam^{(s)}$ are accordingly. 
\end{proof}
\vskip 3mm
\noindent
The next theorem establishes some basic properties of the extremal functional affine surface areas.
\vskip 2mm
\begin{theo} \label{basic}
Let $f \in C_s(\mathbb{R}^n)$. Let $\lam \in \mathbb{R}$. Then we have for the relevant $\lam$-ranges
\par
\noindent
(i) \underline{Invariance} \, For all linear maps $T:\mathbb{R}^n \to \mathbb{R}^n$ such that $\det \, T \neq 0$, 
\begin{equation*}
IS^{(s)}_\lam(f\circ T) =   |\det \, T| ^{2 \lam-1}\, IS^{(s)}_\lam(f),           \ \   \   OS ^{(s)}_\lam(f \circ T)=  |\det \, T| ^{2 \lam-1}\, OS^{(s)} _\lam(f), 
\end{equation*}
\begin{eqnarray*}
{os}^{(s)}_\lam(f \circ T)=   |\det \,T| ^{2 \lam-1}\, {os}^{(s)}_\lam(f),  \ \   \   is^{(s)} _\lam(f \circ T)=  |\det \, T| ^{2 \lam-1}\, is ^{(s)}_\lam(f).
\end{eqnarray*}
Moreover, for all $\alpha \in \mathbb{R}$, $\alpha > 0$, 
$$
IS_\lam^{(s)}( \alpha \, f ) =  \alpha^{1-2 \lam}\, IS_\lam ^{(s)}(f), \hskip 2mm OS_\lam^{(s)}( \alpha \, f ) =  \alpha^{1-2 \lam}\, OS_\lam ^{(s)}(f), 
$$
$$
os_\lam^{(s)}( \alpha \, f ) =  \alpha^{1-2 \lam}\, os_\lam ^{(s)}(f), \hskip 2mm is_\lam^{(s)}( \alpha \, f ) =  \alpha^{1-2 \lam}\, is_\lam ^{(s)}(f).
$$
\vskip 2mm
\noindent
(ii) \underline{Monotonicity in $s$} \, Let $s_1 \leq s_2$. Let $ f \in {\mathit C}_{s_2}(\R^n)$. Then we have for the relevant $\lam$-ranges, 
$$
IS_\lam^{(s_1)} (f) \geq IS_\lam^{(s_2)} (f), \, \hskip 3mm   OS_\lam^{(s_1)} (f) \geq OS_\lam^{(s_2)} (f), 
$$
$$
os_\lam^{(s_1)} (f) \leq os_\lam^{(s_2)} (f), \,  \hskip 3mm  is_\lam^{(s_1)} (f) \leq is_\lam^{(s_2)} (f).
$$
 \vskip 2mm
 \noindent
(iii) \underline{Isoperimetric inequalities}  
 \par
 $\frac{IS_\lam^{(s)}(f)}{\left(\int_{\mathbb{R}^n} f \right)^{1-2\lam}}  \leq \frac{IS_\lam^{(s)}(g^{(s)}_e)}{\left(\int_{\mathbb{R}^n} g^{(s)}_e \right)^{1-2\lam}} 
\hskip 5mm \text {for all} \, \,  \lam \in [0, \frac{1}{2}]$ 
\par
$\frac{OS_\lam^{(s)}(f)}{\left(\int_{\mathbb{R}^n} f \right)^{1-2\lam}} \leq \frac{OS_\lam^{(s)}(g^{(s)}_e)}{\left(\int_{\mathbb{R}^n} g^{(s)}_e \right)^{1-2\lam}}, \hskip 5mm \text {for all} \, \,  \lam \in [\frac{1}{2}, 1]
$
\par
$
\frac{os_\lam^{(s)}(f)}{\left(\int_{\mathbb{R}^n} f \right)^{1-2\lam}} \geq \frac{os_\lam^{(s)}(g^{(s)}_e)}{\left(\int_{\mathbb{R}^n} g^{(s)}_e \right)^{1-2\lam}}, \hskip 5mm \text {for all} \, \,  \lam \in (-\infty, 0].
$
\par
$
\frac{is_\lam^{(s)}(f)}{\left(\int_{\mathbb{R}^n} f \right)^{1-2\lam}} \geq \frac{is_\lam^{(s)}(g^{(s)}_e)}{\left(\int_{\mathbb{R}^n} g^{(s)}_e \right)^{1-2\lam}}, \hskip 5mm \text {for all} \, \,  \lam \in [1, \infty).
$
\par
\noindent
In all inequalities, equality holds trivially when $f=g^{(s)}_e$.
\par
\noindent
Additionally, equality holds in the first  inequality when $\lam=0$ or $\lam=\frac{1}{2}$ and 
equality holds in the second inequality when  $\lam=\frac{1}{2}$ and in the third inequality when $\lam=0$.
\vskip 2mm
\noindent
(iv)   \underline{Duality}   
$$
IS_\lam^{(s)}(f) = OS_{1-\lam}^{(s)}(f_{(s)}^\circ), \hskip 3mm \lam \in [0,1/2],  \hskip 7mm os_\lam^{(s)}(f) = is_{1-\lam}^{(s)}(f_{(s)}^\circ), \hskip 3mm \lam \in (- \infty,0].
$$
\vskip 2mm
\noindent
(v) \underline{Monotonicity in $\lam$}
\par
For $0 \leq \lam \leq \frac{1}{2}$,  $\left(\frac{IS_{\lambda}^{(s)} (f)}{\int_{\mathbb{R}^n} f}\right)^\frac{1}{\lam}$ is  increasing in $\lam$.
\par
For $\frac{1}{2}  \leq \lam \leq  1$,  $\left(\frac{OS_{\lambda}^{(s)} (f)}{\int_{\mathbb{R}^n} f_{(s)}^\circ}\right)^\frac{1}{1-\lam}$ is  decreasing in $\lam$.
\par
For $1 \leq \lam < \infty$,  $\left(\frac{is_{\lambda}^{(s)} (f)}{\int_{\mathbb{R}^n} f}\right)^\frac{1}{\lam}$ is  increasing in $\lam$.
\par
For $- \infty < \lam \leq  0$,  $\left(\frac{os_{\lambda}^{(s)} (f)}{\int_{\mathbb{R}^n} f_{(s)}^\circ}\right)^\frac{1}{1-\lam}$ is  decreasing in $\lam$.
\end{theo}
\vskip 2mm
\noindent
{\bf Remark.}
The following  modified definition of the extremal affine surface areas for bodies, 
$$
IS_p(K) = \max_{L\subset K, L \, \text{is}\,  C^2_+} as_p(L)
$$
- and similarly for the  others -  results in $is_p(K)>0$ and gives a duality relation for all extremal affine surfaces areas for their relevant ranges.
With the original definition, we get duality only for $IS_p$ and $OS_p$ for their relevant ranges. 
\vskip 4mm
\noindent
\subsection{The ``size" of  the extremal affine surface areas.}

From now on, we often write in short $n_s$ instead of $n+\frac{1}{s}$.
\vskip 3mm
\noindent
\subsubsection{The cases: $\frac{1}{s} \in \mathbb{N}$ and $0 < s \leq 1$.} \label{SubS:sinN}

We first give estimates when $\frac{1}{s} \in \mathbb{N}$. We note that these estimates do not follow directly from passing to $K_s(f)$ as our definitions
of extremal affine surface areas differ. 
A consequence of these estimates are estimates for general $s \leq 1$.  When $s=1$, we also recover  from them the results for convex bodies.
The estimates involve the isotropy constant $L_f$ of a function $f$. It's definition  is given in (\ref{Lphi}).
\vskip 2mm
\begin{theo}\label{prelprop}
Let $s>0$ be such that $\frac{1}{s} \in \mathbb{N}$ and let $f \in  {\mathit C}_{s}(\R^n)$. 
\par
\noindent
(i) There is an absolute constant $C>0$ such that for all $\lam \in [0, \frac{1}{2}]$, 
$$ \max \left\{\frac{ C^{n_s \lam}}{n_s^{5/6-\lam} \, L_{f}^{2n\lam}}, \, \, 
\frac{1}{n_s^{n_s(1-2 \lam)}}                                                            \right\} \,     \frac{IS^{(s)}_\lam(g^{(s)}_e)}{\left(\int_{\mathbb{R}^n} g^{(s)}_e dx\right)^{1-2\lam}} 
\leq 
\frac{IS_\lam^{(s)} (f)}{\left(\int_{\mathbb{R}^n} f \right)^{1-2\lam}}
\leq \frac{IS_\lam(g^{(s)}_e)}{\left(\int_{\mathbb{R}^n} g^{(s)}_e dx\right)^{1-2\lam}}.$$
\par
\noindent
Equality holds on the right hand side if $\lam=0$, $\lam=\frac{1}{2}$, and if  $f=g^{(s)}_e$.
\vskip 2mm
\noindent
(ii) There is an absolute  constant $c$ such that for $\lambda \in [\frac{1}{2},1]$, 
\begin{align*}
 \max \left\{ \frac{c^{n_s( 2 \lam-1)}}{n_s^{\lam-1/6} \,    L_{f^\circ_{(s)}}^{2n(1-\lam)}},  \, \,   n_s^{n_s (1 -2 \lam)}\right\} \frac{OS_\lambda^{(s)} (g^{(s)}_e)}{\left(\int_{\mathbb{R}^n}g^{(s)}_e dx\right)^{1-2\lambda}}
 \leq \frac{OS_\lambda^{(s)}(f)}{\left( \int_{\mathbb{R}^n} f dx \right)^{1-2\lambda}}
\leq  \frac{OS_\lambda^{(s)} (g^{(s)}_e)}{\left(\int_{\mathbb{R}^n}g^{(s)}_e dx\right)^{1-2\lambda}}.
\end{align*}
\par
\noindent
Equality holds on the right hand side if  $\lam=\frac{1}{2}$ and if  $f=g^{(s)}_e$.
\vskip 2mm
\noindent
When $f$ is even, we can replace  $n_s^{n_s (1 -2 \lam)}$ with $n_s^{n_s  \frac{1-2 \lam}{2} }$ in (i) and (ii).
\end{theo}
\vskip 2mm
\noindent
{\bf Remarks.}
\par
\noindent
1. It is conjectured that the isotropy constant $L_K$ of a convex body $K$ (see (\ref{LK}) for the definition)  in $\mathbb{R}^n$ is a constant independent of the dimension $n$.
For decades the best known bounds have been $L_K \leq C n^{1/4} \log n$,  proved by Bourgain \cite{Bourgain1991, Bourgain2002}, 
and $L_K \leq C n^{1/4}$, proved by Klartag  \cite{Klartag2006}.  Here, $C$ is an absolute constant. In recent years, based on results by Eldan \cite{Eldan}, there has been much progress on the conjecture due to works by  Chen \cite{Chen}, 
Klartag and Lehec \cite{KlartagLehec}
and Jambulapati, Lee and Vempala \cite{JamLeeVemp}.
The most recent progress is  by Klartag \cite{Klartag2023} who shows that  with a universal constant $C>0$, $L_K \leq C \left( \log n\right) ^\frac{1}{2}$.
Similar remarks hold for $L_f$, see the beginning of subsection  \ref{subsection5.3}.
\vskip 2mm
\noindent
2. We get a slightly better exponent, $n_s^{5/6 - \lam}$,  instead of $n_s^{5/6}$ compared to \cite{GiladiHuangSchuettWerner}. However, a careful analysis of the 
proof of \cite{GiladiHuangSchuettWerner}
shows that also there one obtains $n^{5/6 - p/(n+p)}$.
\vskip 2mm
\noindent
3. If $\lam=1/2$, then the maximum in the lower bound of (i)  and (ii) is achieved for the second term and is $1$. If $\lam= 0$, the maximum in (i) is achieved for the 
first  term and  is equal to $n_s^{-5/6}$. 
If $\lam= 1$, the maximum  in (ii) is achieved for the 
first term and it is equal to $c^{n_s}\, n_s^{-5/6}$. 
\vskip 3mm
\begin{prop}\label{prelprop1}
Let $s>0$ be such that $\frac{1}{s} \in \mathbb{N}$ and let $f \in  {\mathit C}_{s}(\R^n)$. 
\vskip 2mm
\noindent
(i) 
For $\lambda \in (- \infty,0]$, 
\begin{eqnarray*}
\frac{os_\lam^{(s)} (g^{(s)}_e) }{\left(\int_{\mathbb{R}^n}  g^{(s)}_e\, dx \right) ^{1-2 \lam}}\leq \frac{os_\lambda^{(s)}(f)}{\left( \int_{\mathbb{R}^n} f dx \right)^{1-2\lambda}} \leq n_s^{n_s (1-2 \lam)} \frac{os_\lam^{(s)} (g^{(s)}_e) }{\left(\int_{\mathbb{R}^n}  g^{(s)}_e\, dx \right) ^{1-2 \lam}}.
\end{eqnarray*}
\vskip 2mm
\noindent
(ii) Let $\lambda \in [1, \infty)$. Then
\begin{eqnarray*}
\frac{is_\lam^{(s)} (g^{(s)}_e) }{\left(\int_{\mathbb{R}^n}  g^{(s)}_e\, dx \right) ^{1-2 \lam}}\leq \frac{is_\lambda^{(s)}(f)}{\left( \int_{\mathbb{R}^n} f dx \right)^{1-2\lambda}} \leq \frac{1}{n_s^{n_s (1-2 \lam)} }\frac{is_\lam^{(s)} (g^{(s)}_e) }{\left(\int_{\mathbb{R}^n}  g^{(s)}_e\, dx \right) ^{1-2 \lam}}.
\end{eqnarray*}
\vskip 2mm
\noindent
Equality holds on the left hand sides of (i) and (ii)   if  $f=g^{(s)}_e$ and on the left hand side of (i) when $\lam=0$. 
When $f$ is even, we can replace  $n_s^{n_s (1 -2 \lam)}$ with $n_s^{n_s  \frac{1-2 \lam}{2} }$.
\end{prop}

\vskip 3mm
\noindent
The next corollary follows from Theorem \ref{prelprop} resp. Proposition \ref{prelprop1} and the monotonicity property, Theorem \ref{basic} (ii). 
\vskip 2mm
\noindent
\begin{cor}\label{skleiner1}
Let $0<s \leq 1$  and let $f \in  {\mathit C}_{s}(\R^n)$.  Then there are absolute  constants $c$ and $C$ such that the following hold.
\par
\noindent
(i)  Let $\lam \in [0, \frac{1}{2}]$. Then
\begin{eqnarray*}
&& c \frac{ \left(s \, n_s\right)^{\lam(n_s+1)}}{(n+1)^{\lam(n+2)}} \max \left\{\frac{ C^{n \lam}}{(n+1)^{5/6-\lam} \, L_{f}^{2n\lam}}, \, \, 
\frac{1}{(n+1)^{(n+1)(1-2 \lam)}}  \right\} \,     \frac{IS^{(s)}_\lam(g^{(s)}_e)}{\left(\int_{\mathbb{R}^n} g^{(s)}_e dx\right)^{1-2\lam}} \\
&&\leq 
\frac{IS_\lam^{(s)} (f)}{\left(\int_{\mathbb{R}^n} f \right)^{1-2\lam}}
\leq \frac{IS_\lam(g^{(s)}_e)}{\left(\int_{\mathbb{R}^n} g^{(s)}_e dx\right)^{1-2\lam}}.
\end{eqnarray*}
\vskip 2mm
\noindent
(ii) Let $\lambda \in [\frac{1}{2},1]$. Then
\begin{eqnarray*}
&&c  \frac{ \left(s \, n_s\right)^{\lam(n_s+1)}}{(n+1)^{\lam(n+2)}} \max \left\{ \frac{C^{n (2 \lam-1)}}{(n+1)^{\lam-1/6} \,    L_{f^\circ_{(s)}}^{2n(1-\lam)}},  \, \,   (n+1)^{(n+1) (1 -2 \lam)}\right\} \frac{OS_\lambda^{(s)} (g^{(s)}_e)}{\left(\int_{\mathbb{R}^n}g^{(s)}_e dx\right)^{1-2\lambda}}\\
&& \leq \frac{OS_\lambda^{(s)}(f)}{\left( \int_{\mathbb{R}^n} f dx \right)^{1-2\lambda}}
\leq  \frac{OS_\lambda^{(s)} (g^{(s)}_e)}{\left(\int_{\mathbb{R}^n}g^{(s)}_e dx\right)^{1-2\lambda}}.
\end{eqnarray*}
\vskip 2mm
\noindent
(iii) Let $\lambda \in (- \infty,0]$. Then
\begin{eqnarray*}
\frac{os_\lam^{(s)} (g^{(s)}_e) }{\left(\int_{\mathbb{R}^n}  g^{(s)}_e\, dx \right) ^{1-2 \lam}}\leq \frac{os_\lambda^{(s)}(f)}{\left( \int_{\mathbb{R}^n} f dx \right)^{1-2\lambda}} \leq 
 \frac{c^\lam \,  \left(s \, n_s\right)^{\lam(n_s+1)}}{(n+1)^{(n+1) (3 \lam-1) + \lam}}   
  \frac{os_\lam^{(s)} (g^{(s)}_e) }{\left(\int_{\mathbb{R}^n}  g^{(s)}_e\, dx \right) ^{1-2 \lam}}.
\end{eqnarray*}
\vskip 2mm
\noindent
(iv) Let $\lambda \in [1, \infty)$. Then
\begin{eqnarray*}
\frac{is_\lam^{(s)} (g^{(s)}_e) }{\left(\int_{\mathbb{R}^n}  g^{(s)}_e\, dx \right) ^{1-2 \lam}}\leq \frac{is_\lambda^{(s)}(f)}{\left( \int_{\mathbb{R}^n} f dx \right)^{1-2\lambda}} \leq
 c^\lam \frac{ \left(s \, n_s\right)^{\lam(n_s+1)}}{(n+1)^{n(1-\lam) +1}}   
\frac{is_\lam^{(s)} (g^{(s)}_e) }{\left(\int_{\mathbb{R}^n}  g^{(s)}_e\, dx \right) ^{1-2 \lam}}.
\end{eqnarray*}
\vskip 2mm
\noindent
Statements about equality cases and even functions are as in Theorem \ref{prelprop} and Proposition \ref{prelprop1}.
\end{cor}
\vskip 3mm
\begin{proof}
By the monotonicity property Theorem \ref{basic} (ii), Theorem \ref{prelprop}, Lemmas \ref{ge-prop} and 
\ref{lemma-phie}, 
\begin{eqnarray*}
&&\frac{IS_\lambda^{(s)}(f)}{\left( \int_{\mathbb{R}^n} f dx \right)^{1-2\lambda}} \geq \frac{IS_\lambda^{(1)}(f)}{\left( \int_{\mathbb{R}^n} f dx \right)^{1-2\lambda}}
\geq  \\
&& \max \left\{ \frac{c^{(n+1)\lam}}{(n+1)^{5/6 -\lam} \,    L_{f}^{2n \lam}},  \, \, \frac{1}{  (n+1)^{(n+1) (1 -2 \lam)}}\right\} \frac{IS_\lambda^{(s)} (g_e^{(1)})}{\left(\int_{\mathbb{R}^n}g_e^{(1)} dx\right)^{1-2\lambda}} =\\
&&  \max \left\{ \frac{c^{(n+1) \lam}}{(n+1)^{5/6-\lam} \,    L_{f}^{2n \lam}},  \, \,  \frac{1}{ (n+1)^{(n+1) (1 -2 \lam)}}\right\} \frac{IS_\lambda^{(s)} (g_e^{(s)})}{\left(\int_{\mathbb{R}^n}g_e^{(s)} dx\right)^{1-2\lambda}}  \left( \frac{\int_{\mathbb{R}^n}g_e^{(1)} dx}{\int_{\mathbb{R}^n}g_e^{(s)} dx}\right)^{2 \lam}.
\end{eqnarray*}
Now we observe that with an absolute constant $c$, 
$$
\frac{\int_{\mathbb{R}^n}g_e^{(1)} dx}{\int_{\mathbb{R}^n}g_e^{(s)} dx} = c \frac{ \left(s \, n_s\right)^{\frac{1}{2}(n_s+1)}}{(n+1)^{\frac{1}{2}(n+2)}}.
$$
This finishes the proof of (i). The other estimates are done in the same way.
\end{proof}
\vskip 3mm
\noindent
{\bf Remark.}
\par
\noindent
As an immediate consequence of the previous corollary, we get estimates for  log-concave functions   $f=e^{-\psi}$ where  
$\psi \colon \R^n \to \R \cup\{+\infty\}$ is a convex function. Log-convave functions correspond  to the case $s=0$.  However, these estimates are far from optimal.
The estimates follow immediately from Corollary \ref{skleiner1}, letting $s \to 0$,  and observing that 
 $\lim_{s \to 0} \left(s \, n_s\right)^{\lam(n_s+1)} = c^n$, for some constant $c>0$. We only list $IS_\lam^{(0)}$. The others are accordingly.
\par
\noindent
Let  $f \in  {\mathit C}_{0}(\R^n)$.  Then there are absolute  constants $c$ and $C$ such that 
for $\lam \in [0, \frac{1}{2}]$, 
\begin{eqnarray*}
&& \frac{ c^{n \lam }}{(n+1)^{\lam(n+2)}} \max \left\{\frac{ C^{n \lam}}{(n+1)^{5/6-\lam} \, L_{f}^{2n\lam}}, \, \, 
\frac{1}{(n+1)^{(n+1)(1-2 \lam)}}  \right\} \,     \frac{IS^{(0)}_\lam(g_e)}{\left(\int_{\mathbb{R}^n} g^{(s)}_e dx\right)^{1-2\lam}} \\
&&\leq 
\frac{IS_\lam^{(0)} (f)}{\left(\int_{\mathbb{R}^n} f \right)^{1-2\lam}}
\leq \frac{IS_\lam^{(0)}(g_e)}{\left(\int_{\mathbb{R}^n} g^{(s)}_e dx\right)^{1-2\lam}}, 
\end{eqnarray*}
where $g^{(0)}_{e}(x) = e^{-\frac{|x|^2}{2}}$.
Statements about equality cases and even functions are as in Theorem \ref{prelprop} and Proposition \ref{prelprop1}.
\vskip 4mm
\noindent
Now we state the general cases.  We already have estimates in the general case  for $s \leq 1$ from Corollary \ref{skleiner1}. We  give  estimates when $s \geq 1$.
\vskip 2mm
\noindent
\subsubsection{The case $ s \geq 1$.}
\vskip 2mm
\noindent
\begin{theo} \label{estimate} Let $f \in  {\mathit C}_{s}(\R^n)$. 
\par
\noindent
(i) Let  $\lam \in [ \frac{1}{2},1]$. 
\par
\noindent
Then we have for all $s \in \mathbb{R}$, $s>0$:   \hskip 6mm
$
\frac{OS_\lambda^{(s)} (f)}{\left(\int_{\mathbb{R}^n}f dx\right)^{1-2\lambda}} \leq \frac{OS_\lambda^{(s)} (g^{(s)}_e)}{\left(\int_{\mathbb{R}^n}g^{(s)}_e dx\right)^{1-2\lambda}}.
$
\par
\noindent
Equality holds on the right hand side if $\lam=\frac{1}{2}$, and if  $f=g^{(s)}_e$.
\par
\noindent
There is an absolute constant $c>0$ such  that for all $s \geq 1$,
$$
 \frac{OS_\lam^{(s)} (f)}{\left(\int_{\mathbb{R}^n}f \right)^{1-2\lam}}  \geq \frac{ c^{n (2\lam-1)}}{(n+1)^{(3n/2+1)(2\lam-1)}}
 \frac{OS_\lam^{(s)}(g^{(s)}_e)}{\left(\int_{\mathbb{R}^n} g^{(s)}_e dx\right)^{1-2\lam}}
$$
\vskip 2mm
\noindent
(ii) Let  $\lam \in [0, \frac{1}{2}]$.
\par
\noindent
Then we have for all $s \in \mathbb{R}$, $s>0$: \hskip 6mm
$ \frac{IS_\lam^{(s)} (f)}{\left(\int_{\mathbb{R}^n}f \right)^{1-2\lam}}  \leq \frac{IS_\lam^{(s)}(g^{(s)}_e)}{\left(\int_{\mathbb{R}^n} g^{(s)}_e dx\right)^{1-2\lam}}.$
\par
\noindent
Equality holds  if $\lam=0$, $\lam=\frac{1}{2}$, and if  $f=g^{(s)}_e$.
\par
\noindent
There is an absolute constant $c>0$ such  that for all $s \geq 1$,
$$
 \frac{IS_\lam^{(s)} (f)}{\left(\int_{\mathbb{R}^n}f \right)^{1-2\lam}}  \geq \frac{ c^{n(2\lam-1)} }{(n+1)^{(3n/2+1)(1-2\lam)}}
 \frac{IS_\lam^{(s)}(g^{(s)}_e)}{\left(\int_{\mathbb{R}^n} g^{(s)}_e dx\right)^{1-2\lam}}
$$
\par
\noindent
If $f$ is even, $(n+1)^{(3n/2+1)(1-2\lam)}$ can be replaced in (i) and (ii) by $(n+1)^{\frac{(3n/2+1)}{2}(1-2\lam)}$.

\end{theo}
\vskip 4mm
\noindent
\begin{prop}\label{estimate-prop}
Let $s \geq 1$.  Let $f \in  {\mathit C}_{s}(\R^n)$. 
 \par
 \noindent
(i) Let $\lambda \in (-\infty,0]$. Then there is  an absolute constant $c>0$ such that 
$$\frac{os_\lambda^{(s)} (g^{(s)}_e)}{\left(\int_{\mathbb{R}^n}g^{(s)}_e dx\right)^{1-2\lambda}} \leq  \frac{os_\lambda^{(s)} (f)}{\left(\int_{\mathbb{R}^n}f dx\right)^{1-2\lambda}} 
\leq c^{n(1- 2 \lam)} (n+1)^{(3n/2+1) (1-2\lambda)} \frac{os_\lambda^{(s)} (g^{(s)}_e)}{\left(\int_{\mathbb{R}^n}g^{(s)}_e dx\right)^{1-2\lambda}}.$$
Equality holds on the left hand side if $\lam=0$ or  if  $f=g^{(s)}_e$.
\vskip 2mm
\noindent
(ii) Let $\lambda \in [1, \infty)$. Then there is  an absolute constant $c>0$ such that
$$\frac{is_\lambda^{(s)} (g^{(s)}_e)}{\left(\int_{\mathbb{R}^n}g^{(s)}_e dx\right)^{1-2\lambda}} \leq  \frac{is_\lambda^{(s)} (f)}{\left(\int_{\mathbb{R}^n}f dx\right)^{1-2\lambda}} 
\leq c^{n(2 \lam-1)} (n+1)^{(3n/2+1) (2\lambda-1)}  \frac{is_\lambda^{(s)} (g^{(s)}_e)}{\left(\int_{\mathbb{R}^n}g^{(s)}_e dx\right)^{1-2\lambda}}.$$
Equality holds on the left hand side  if  $f=g^{(s)}_e$.
\par
\noindent
If $f$ is even, $(n+1)^{(3n/2+1)(1-2\lam)}$ can be replaced  in (i) and (ii) by $(n+1)^{\frac{(3n/2+1)}{2}(1-2\lam)}$.
\end{prop}

\vskip 3mm
\subsection{Results for convex bodies.}
As a corollary to the theorems of the previous subsection, we recover  results of \cite{GiladiHuangSchuettWerner} on estimates for the size of extremal affine surface areas for convex bodies.
We can assume without loss of generality (see also \cite{GiladiHuangSchuettWerner})  that  $0$ is the center of 
gravity of $K$. We recover and improve  the results of \cite{GiladiHuangSchuettWerner} for those convex bodies that
are symmetric about a hyperplane through $0$. This is  a large class, including all unconditional bodies and all bodies obtained form a general convex
body by one Steiner symmetrization.

\vskip 2mm
\begin{theo}\label{bodies}
Let $K$ be a convex body in $\mathbb{R}^{n+1}$ that is symmetric about a hyperplane through $0$. Let $p \in \mathbb{R}$, $p \neq -(n+1)$. Then 
\par
(i) For $p \in [0, n+1]$, 
\begin{eqnarray*} 
 &&\max \left\{\left(\frac{c}{L_K} \right)^\frac{2(n+1)p}{n+1+p} \frac{1}{ n^{5/6- p/(n+1+p)} }, (n+1) ^{-(n+1)\frac{n+1-p}{n+1+p}}\right\} \,  (n+1) \, |B^{n+1}_2|^ \frac{2p}{n+1+p}\\
&& \hskip 70mm \leq \frac{IS_p(K)}{|K|^\frac{n+1-p}{n+1+p} }  \leq  (n+1) \, |B^{n+1}_2|^ \frac{2p}{n+1+p}.
\end{eqnarray*}
\par
(ii) For $p \in [n+1, \infty]$, 
\begin{eqnarray*} 
 &&\max \left\{\left(\frac{c}{L_{K^\circ}} \right)^\frac{2(n+1)p}{n+1+p} \frac{1}{ n^{5/6- p/(n+1+p)} }, (n+1) ^{-(n+1)\frac{n+1-p}{n+1+p}}\right\} \,  (n+1) \, |B^{n+1}_2|^ \frac{2p}{n+1+p}\\
&& \hskip 70mm \leq \frac{OS_p(K)}{|K|^\frac{n+1-p}{n+1+p} }  \leq  (n+1) \, |B^{n+1}_2|^ \frac{2p}{n+1+p}.
\end{eqnarray*}
\par
(iii) For $p \in (- (n+1),  0]$, 
\begin{eqnarray*} 
   (n+1) \, |B^{n+1}_2|^ \frac{2p}{n+1+p}  \leq \frac{os_p(K)}{|K|^\frac{n+1-p}{n+1+p} }  \leq  (n+1) ^{(n+1)\frac{n+1-p}{n+1+p}} (n+1) \, |B^{n+1}_2|^ \frac{2p}{n+1+p}.
\end{eqnarray*}
\par
(iv) For $p \in [- \infty,  -(n+1))$, 
\begin{eqnarray*} 
   (n+1) \, |B^{n+1}_2|^ \frac{2p}{n+1+p}  \leq \frac{is_p(K)}{|K|^\frac{n+1-p}{n+1+p} }  \leq  (n+1) ^{(n+1)\frac{n+1-p}{n+1+p}} (n+1) \, |B^{n+1}_2|^ \frac{2p}{n+1+p}.
\end{eqnarray*}
\end{theo}
\vskip 2mm
\begin{proof}
We show how to get (i) as a consequence of Theorem \ref{prelprop} (i) for $s=1$.  Similarly one gets (ii) from Theorem \ref{prelprop} (ii) and (iii) and (iv) from 
Proposition \ref{prelprop1}.
The inequality from above is just a consequence of the affine isoperimetric inequality.
For the inequality from below, we note that we can assume without loss of generality that  $K \subset \mathbb{R}^{n+1}$  is symmetric about  $\mathbb{R}^n$.
Then $\partial K$ is described by a $1$-concave function $f$ in the upper halfplane and by $-f$ in the lower halfplane.
Thus $|K|= |K_1(f)|= 2 \int_{\mathbb{R}^n} f dx $.  Moreover, 
$$
IS_\lam^{(s)} (g_e^{(1)}) = OS_\lam^{(1)} (g_e^{(1)})=is_\lam^{(1)} (g_e^{(1)})=os_\lam^{(1)} (g_e^{(1)})
$$
and (i) follows from Theorem \ref{prelprop} (i), as $\lam=\frac{p}{n+1+p}$.
\end{proof}

\vskip 4mm
\noindent

\section{Proofs}\label{proofs}

\subsection{Proof of Theorem \ref{ranges}.}
\begin{proof}
(i) {\bf The case $OS_\lam^{(s)}(f)$} \label{OS}
\par
\noindent
If  $\lam=\frac{1}{2}$, then  it follows from Lemma \ref{lemma-phie} that
\begin{eqnarray*}
OS_{\frac{1}{2}}^{(s)}(f) &=&  \sup_{h \geq f, h \in {\mathit C}^+_{s}(\R^n)} as_{\frac{1}{2}} ^{(s)}(h)  \geq  
\sup_{g^{(s)}_{e,R}  \geq f} as_{\frac{1}{2}} ^{(s)}(g^{(s)}_{e,R})=
 as_{\frac{1}{2}}^{(s)} (g^{(s)}_e) = \int_{\mathbb{R}^n} g^{(s)}_e \, dx.
\end{eqnarray*}
Note that a function $g^{(s)}_{e,R}  \geq f$ exists as $f$ is bounded and $S_f $ is bounded.
On the other hand, 
by Proposition \ref{Ungleichung1} and using again the above  mentioned form of the functional Blaschke-Santal\'o inequality (\ref{Bl-Sa}), we get
\begin{eqnarray*}
OS_{\frac{1}{2}}^{(s)}(f) &=& \sup_{h \geq f, h \in {\mathit C}^+_{s}(\R^n)} as_{\frac{1}{2}} ^{(s)}(h) 
\leq \sup_{h \geq f, h\in C^+_{s}(\mathbb{R}^n)}   \left(\int_{\R^n} h \  dx \right)^{\frac{1}{2}}
 \left( \int_{\R^n} h^\circ_{(s)}  \  dx \right)^\frac{1}{2}\\
 &\leq& \int_{\R^n}   \left(1- s |x|^2\right)^\frac{1}{2 s}\, dx =  \int_{\mathbb{R}^n} g^{(s)}_e \, dx.
 \end{eqnarray*}
All together
$$
OS_{\frac{1}{2}}^{(s)}(f) =\int_{\mathbb{R}^n} g^{(s)}_e \, dx = OS_{\frac{1}{2}}^{(s)}\left(g^{(s)}_e \right).
$$
\par
\noindent
For $\lam=1$, we get by (\ref{def:asa-1}), and  as $h \geq f$ is equivalent to $h_{(s)}^\circ \leq f_{(s)}^\circ$,  
\begin{eqnarray*}
OS_{1}^{(s)}(f)  
&=& \sup_{h \geq f, h \in {\mathit C}^+_{s}(\R^n)} as_{1} ^{(s)}(h) =  \sup_{h_{(s)}^{\circ}\leq f_{(s)}^{\circ}, h \in {\mathit C}^+_{s}(\R^n)} \int_{S_{f_{(s)}^{\circ}}} h_{(s)}^{\circ} (x) dx\\
&=&
\int_{S_{f_{(s)}^{\circ}}} f_{(s)}^{\circ} (x) dx >0.
\end{eqnarray*}
\par
\noindent
If $0 \leq  \lam < \frac{1}{2}$,  then for all $f$,  $OS_\lam^{(s)}(f) = \infty$.
Indeed, by Lemma \ref{lemma-phie}, and the fact that $1-2\lambda >0$, 
\begin{eqnarray*}
OS_\lam^{(s)}(f)= \sup_{h \geq f, h \in {\mathit C}^+_{s}(\R^n)} as_\lambda^{(s)}(h) \geq   \sup_{ g^{(s)}_{e,R} \geq f} as_\lam^{(s)}(g^{(s)}_{e,R}) = \sup_{R } R^{(n+\frac{1}{s})(1-2\lambda)} \int_{\mathbb{R}^n} g^{(s)}_{e}\, dx  = \infty.
\end{eqnarray*}
\par
\noindent
If $- \infty < \lam <0$,  then for all $f$,  $OS_\lam^{(s)}(f) = \infty$.
By Lemma \ref{lemma-phie}, 
\begin{eqnarray*}
OS_\lam^{(s)}(f) \geq   \sup_{  g^{(s)}_{e,R} \geq f} as_\lam^{(s)}( g^{(s)}_{e,R} ) = \sup_{R } R^{(n+\frac{1}{s})(1-2\lambda)} \int_{\mathbb{R}^n} g^{(s)}_{e}\, dx  = \infty.
\end{eqnarray*}
If $1 < \lam <  \infty$,  then for all $f$,  $OS_\lam(f) = \infty$. 
\par
\noindent
By definition, 
$OS_\lam^{(s)}(f)= \sup_{h \geq f, h \in {\mathit C}^+_{s}(\R^n)} as_\lambda^{(s)}(h)$. 
For $\varepsilon , R$, we now consider a specific $h\in  {\mathit C}_s^+$, namely, $h_{R,\varepsilon}$ such that 
\begin{align}\label{part1!concrete}
h_{R, \varepsilon}^s(x) = R-(\varepsilon + |x|^2)^\frac{1}{2}.
\end{align} 
We have that $S(h_{R, \varepsilon}) = (R^2-\varepsilon)^\frac{1}{2} B^n_2$.  We choose $R$ and $\varepsilon$ so that $S(f) \subset S(h_{R, \varepsilon})$ and such that 
$h _{R, \varepsilon}\geq f$.
Then
\begin{align*}
h_{R, \varepsilon}^s(x) - \langle x, \nabla h_{R, \varepsilon}^s(x)\rangle = \frac{R(\varepsilon + |x|^2)^\frac{1}{2} - \varepsilon}{(\varepsilon +|x|^2)^\frac{1}{2}}
\end{align*}
and 
 \begin{align*}
\det (-\nabla^2 h_{R, \varepsilon}^s(x)) = \frac{\det (A_\varepsilon)}{(\varepsilon +|x|^2)^\frac{3n}{2}}, 
\end{align*}
where 
 \begin{align*}\label{part3!concrete}
A_\varepsilon = 
\begin{bmatrix}
\varepsilon + |x|^2 - x_1^2 & x_1x_2 & \cdots & x_1x_n \\
x_2x_1 & \varepsilon + |x|^2-x_2^2 & \cdots & x_2x_n\\
\vdots & & \ddots & \vdots \\
x_nx_1& x_nx_2 & \cdots & \varepsilon +|x|^2 -x_n^2
\end{bmatrix}.
\end{align*}
Hence, for $R, \varepsilon$ such that $(R/2)^2 \geq 2 \varepsilon$, which implies that  $(R^2 - \varepsilon )^\frac{1}{2} \geq \varepsilon^\frac{1}{2}$, we get
\begin{eqnarray*}
OS_\lam^{(s)}(f)&=&\sup_{h \geq f, h \in {\mathit C}^+_{s}(\R^n)} as_\lambda^{(s)}(h) \geq \sup_{R, \varepsilon} \, as_\lambda^{(s)}(h_{R, \varepsilon})\\
&\geq & \sup_{R, \varepsilon}\frac{1}{s^{n\lambda}(1+ns)}\int_{\sqrt{\varepsilon}B_2^n}\frac{(R-(\varepsilon+|x|^2)^\frac{1}{2})^{\left(1-\frac{1}{s}\right)(\lambda - 1)}(\det (A_\varepsilon))^\lambda}{(\varepsilon+|x|^2)^\frac{3n\lambda}{2}\frac{(R(\varepsilon+|x|^2)^\frac{1}{2}-\varepsilon)^{\lambda\left( n+\frac{1}{s} +1\right) - 1}}{(\varepsilon + |x|^2)^{\frac{1}{2}\left(\lambda\left( n+\frac{1}{s}+1\right)-1\right)}}}dx.
\end{eqnarray*}
Now we use that $\frac{R(\varepsilon+|x|^2)^\frac{1}{2} - \varepsilon}{(\varepsilon + |x|^2)^{\frac{1}{2}}} \leq R$. This gives
\begin{eqnarray*}
OS_\lam^{(s)}(f)&\geq& \sup_{R, \varepsilon} \frac{R^{-(\lambda (n+\frac{1}{s}+1)-1)}}{s^{n\lambda}(1+ns)}\int_{\sqrt{\varepsilon}B_2^n} 
\frac{\left(R-\varepsilon^\frac{1}{2} \left(1+\left(\frac{|x|}{\sqrt{\varepsilon}}\right)^2\right)^\frac{1}{2}\right)^{\left(1-\frac{1}{s}\right)(\lambda - 1)} \varepsilon^{n \lam}(\det (A))^\lambda}{\varepsilon^\frac{3 n \lam}{2} \left(1+\left(\frac{|x|}{\sqrt{\varepsilon}}\right)^2\right)^\frac{3n\lambda}{2}} dx, 
\end{eqnarray*}
where 
 \begin{align*}
A = 
\begin{bmatrix}
1 + \left(\frac{|x|}{\sqrt{\varepsilon}}\right)^2 - \left(\frac{x_1}{\sqrt{\varepsilon}}\right)^2 & \frac{x_1x_2}{\varepsilon} & \cdots & \frac{x_1x_n}{\varepsilon}\\
\frac{x_2 x_1}{\varepsilon} & 1 + \left(\frac{|x|}{\sqrt{\varepsilon}}\right)^2 - \left(\frac{x_2}{\sqrt{\varepsilon}}\right)^2& \cdots & \frac{x_2 x_n}{\varepsilon}\\
\vdots & & \ddots & \vdots \\
\frac{x_n x_1}{\varepsilon}&\frac{x_n x_2}{\varepsilon} & \cdots & 1 + \left(\frac{|x|}{\sqrt{\varepsilon}}\right)^2 - \left(\frac{x_n}{\sqrt{\varepsilon}}\right)^2
\end{bmatrix}.
\end{align*}
With the change of variable $\frac{x_i}{\sqrt{\varepsilon}}= y_i$, this becomes
\begin{eqnarray*}
OS_\lam^{(s)}(f)&\geq& \sup_{R, \varepsilon} \frac{R^{-(\lambda (n+\frac{1}{s}+1)-1)}}{s^{n\lambda}(1+ns) \varepsilon^{\frac{n}{2}(\lam-1)}}\int_{B_2^n} 
\frac{\left(R-\varepsilon^\frac{1}{2} \left(1+|y|^2\right)^\frac{1}{2}\right)^{\left(1-\frac{1}{s}\right)(\lambda - 1)} (\det (A))^\lambda}{ \left(1+\left(|y| \right)^2\right)^\frac{3n\lambda}{2}} dy.
\end{eqnarray*}
For all $s \leq 1$ we thus get
\begin{eqnarray*}
OS_\lam^{(s)}(f)&\geq& 
OS_\lam^{(1)}(f)\geq \sup_{R, \varepsilon} \frac{R^{-(\lambda (n+2)-1)}}{(1+n) \varepsilon^{\frac{n}{2}(\lam-1)}}\int_{B_2^n} 
\frac{ (\det (A))^\lambda}{ \left(1+|y|^2\right)^\frac{3n\lambda}{2}} dy\\
&\geq&\frac{1}{n+1} \left( \int_{B_2^n} 
\frac{ (\det (A))^\lambda}{ \left(1+|y|^2\right)^\frac{3n\lambda}{2}} dy \right) \, \sup_{R, \varepsilon} \frac{1}{R^{\lambda (n+2)-1} \varepsilon^{\frac{n}{2}(\lam-1)}} = \infty,
\end{eqnarray*}
if we choose e.g.,  $\varepsilon = R^{- \frac{2 \lam(n+2)}{(\lam-1)(n-1)}}$.
If $s>1$, then, also choosing $R, \varepsilon$ such  that $R/2 \geq (2 \varepsilon)^{(1/2)}$, we get 
\begin{eqnarray*}
OS_\lam^{(s)}(f)&\geq& \frac{1}{s^{n\lambda}(1+ns)}  \int_{B_2^n} 
\frac{ (\det (A))^\lambda}{ \left(1+|y|^2\right)^\frac{3n\lambda}{2}} dy  \, \sup_{R, \varepsilon}  \frac{\left(R/2 \right)^{\left(1-\frac{1}{s}\right)(\lambda - 1)}}{ R^{\lambda (n+\frac{1}{s}+1)-1}\, \varepsilon^{\frac{n}{2}(\lam-1)}} = \infty,
\end{eqnarray*}
with an adequate  choice for $R, \varepsilon$.
\par
\noindent
Let now $\lam \in (1/2,1)$. By Proposition \ref{Ungleichung1}, we get
\begin{eqnarray*}
OS_\lam^{(s)}(f)&= &\sup_{h \geq f, h \in {\mathit C}^+_{s}(\R^n)} as_\lambda^{(s)}(h) \leq \sup_{h \geq f, h \in {\mathit C}^+_{s}(\R^n)} \left(\int_{\R^n}  h\, dx \right)^{1-\lam}
 \left( \int_{\R^n}  h_{(s)}^\circ  \  dx \right)^{ \lam} \\
 &=&  \sup_{h \geq f, h \in {\mathit C}^+_{s}(\R^n)} \left(\int_{\R^n}  h\, dx \right)^{1- 2\lam}
 \left(\int_{\R^n}  h\, dx \right)^{\lam}  \left( \int_{\R^n}  h_{(s)}^\circ  \  dx \right)^{ \lam}.
\end{eqnarray*}
Now we use the estimate (\ref{Bl-Sa}). Altogether we get that
\begin{eqnarray*}
OS_\lam^{(s)}(f)
 &\leq&  \left(\int_{\R^n} f \,  dx \right)^{1-2 \lam} \left(\int_{\R^n}  g^{(s)} _e dx\right)^{2 \lam} < \infty.
\end{eqnarray*}
On the other hand, let $R>0$ be such that $g^{(s)}_{e,R} \geq f$. Then
\begin{eqnarray*}
OS_\lam^{(s)}(f) \geq   as_\lam^{(s)}(g^{(s)}_{e,R}) = R^{(n+\frac{1}{s})(1-2\lambda)} \int_{\mathbb{R}^n} g^{(s)}_{e}\, dx  >0.
\end{eqnarray*}
Thus $0 < OS_\lam^{(s)}(f) < \infty $ on $\lam \in (1/2,1)$ and thus:
\par
\noindent
{\bf Conclusion}.  
The relevant $\lam$-range for $OS_\lam^{(s)}$ is $\lam \in [\frac{1}{2},1]$. 
\vskip 2mm
\noindent
(ii) {\bf The case $IS_\lam^{(s)}(f)$}. \label{IS}
\par
\noindent
Let $\lam=0$. By (\ref{def:asa-0}), 
\begin{eqnarray*}
IS_0^{(s)}(f) &=&  \sup_{h \leq f, h \in {\mathit C}^+_{s}(\R^n)}  as_{0} ^{(s)}(h) = \sup_{h \leq f, h \in {\mathit C}^+_{s}(\R^n)}   \int_{\mathbb{R}^n} h \, dx =  \int_{\mathbb{R}^n} f \, dx.
\end{eqnarray*}
\par
\noindent
If  $\lam=\frac{1}{2}$,
it follows from 
Proposition \ref{Ungleichung1}  that 
\begin{eqnarray*}
IS_{\frac{1}{2}}^{(s)}(f) &=&   \sup_{h \leq f, h \in C^+_{s}} as_{\frac{1}{2}} ^{(s)}(h) 
\leq \sup_{h \leq f, h\in C^+_{s}(\mathbb{R}^n)}   \left(\int_{\R^n} h \  dx \right)^{\frac{1}{2}}
 \left( \int_{\R^n} h^\circ_{(s)}  \  dx \right)^\frac{1}{2}.
 \end{eqnarray*}
We  now 
apply again the above general form of the functional Blaschke-Santal\'o inequality (\ref{concsantalo}), again 
with $f=h$, $g=h^\circ _{(s)}$ and $\rho(t) = (1-st)^\frac{1}{2s}$. By definition (\ref{Leg})  of $h^\circ _{(s)}$,
$$
h(x) h^\circ _{(s)}(y) \leq \left(1- s |x|^2\right)^\frac{1}{s}.
$$
Therefore
$$
\int_{\R^n} h \, dx  \int_{\R^n} h^\circ _{(s)} \, dx   \leq \left(\int_{\R^n}   \left(1- s |x|^2\right)^\frac{1}{2 s}\, dx\right)^2
$$  
and hence
\begin{eqnarray*}
IS_{\frac{1}{2}}^{(s)}(f) \leq  \int_{\mathbb{R}^n} g^{(s)}_e \, dx.
\end{eqnarray*}
On the other hand by Lemma \ref{lemma-phie}, 
\begin{eqnarray*}
IS_{\frac{1}{2}}^{(s)}(f) &\geq&   \sup_{g^{(s)}_{e,r} \leq f} as_{\frac{1}{2}} ^{(s)}(g^{(s)}_{e,r}) = as_{\frac{1}{2}} ^{(s)}(g^{(s)}_{e}) = \int_{\mathbb{R}^n} g^{(s)}_e \, dx.
\end{eqnarray*}
Functions $g^{(s)}_{e,r} \leq f$ exist,  as $0 \in \text{int}(S_f)$ and $f(0) >0$.
All together
$$
IS_{\frac{1}{2}}^{(s)}(f) =\int_{\mathbb{R}^n} g^{(s)}_e \, dx = IS_{\frac{1}{2}}^{(s)}\left(g^{(s)}_e \right).
$$
\par
\noindent
If $\frac{1}{2} < \lam $, then $IS_\lam^{(s)}(f) = \infty$. By Proposition \ref{properties} and Lemma \ref{lemma-phie}, 
\begin{eqnarray*}
IS_\lam^{(s)}(f) \geq   \sup_{\alpha g^{(s)}_{e,r}  \leq f } as_\lam^{(s)}(\alpha g^{(s)}_{e,r}) =  \int _{\mathbb{R}^n}g^{(s)}_e dx \sup_{\alpha g^{(s)}_{e,r}  \leq f } \alpha^{1-2\lam} r^{(n+\frac{1}{s})(1-2 \lam)}= \infty,
\end{eqnarray*}
\par
\noindent
If $- \infty < \lam <0$,  then for all $f$,  $IS_\lam^{(s)}(f) = \infty$.
Indeed, by Proposition \ref{properties} (iii), 
\begin{eqnarray*}
IS_\lam^{(s)}(f) &= & \sup_{h \in C^+_{s},  h \leq f} as_\lam^{(s)}(h) =  \sup_{h \in C^+_{s},  h \leq f } as_{1-\lam}^{(s)}(h_{(s)}^\circ) =\sup_{h_{(s)}^\circ \in C^+_{s},  h_{(s)}^\circ \geq f_{(s)}^\circ } as_{1-\lam}^{(s)}(h_{(s)}^\circ) \\
&= &OS^{(s)}_{1-\lam}(f_{(s)}^\circ) = \infty.
\end{eqnarray*}
The last identity holds by the range of $OS_\lam^{(s)}$.
\par
\noindent
As $0 < OS_\lam^{(s)}(f) < \infty $ on $\lam \in (1/2,1)$, it follows by duality that $0 < IS_\lam^{(s)}(f) < \infty $ on $(0, 1/2)$. 
\par
\noindent
{\bf Conclusion}. 
The relevant $\lam$-range for $IS_\lam^{(s)}$ is $\lam \in [0,\frac{1}{2}]$. \\
\vskip 2mm
\noindent
(iii) {\bf The cases $is_\lam^{(s)}(f)$ and $os_\lam^{(s)}(f)$.} \label{os}
\par
\noindent
We treat these two cases together.
\par
\noindent
For $\lam=1$, we get by Proposition \ref{Ungleichung1}, 
\begin{eqnarray*}
is_{1}^{(s)}(f) =  \inf_{h \in C_s^+(\mathbb{R}^n), h \leq f} as_{1} ^{(s)}(h) = \inf_{h \in C_s^+(\mathbb{R}^n), h_{(s)}^\circ \geq f_{(s)}^\circ} as_{0} ^{(s)}(h_{(s)}^\circ) =
os_{0}^{(s)}(f_{(s)}^\circ) = \int_{\mathbb{R}^n }  f_{(s)}^\circ dx.
\end{eqnarray*}
\par
\noindent
If  $\lam=0$, then $os_{0}^{(s)}(f) = \int_{\mathbb{R}^n} f dx$. Indeed, 
\begin{eqnarray*}
os_{0}^{(s)}(f) &=&  \inf_{h \geq f, h \in {\mathit C}^+_{s}(\R^n)} as_0^{(s)}(h) = \inf_{h \geq f, h \in {\mathit C}^+_{s}(\R^n)} \int_{\mathbb{R}^n} h dx 
= \int_{\mathbb{R}^n}  f dx.
\end{eqnarray*}
\par
\noindent
Let $ - \infty < \lam < \frac{1}{2}$. Then, again with Proposition \ref{properties},   
\begin{eqnarray*}
0 \leq is_\lam^{(s)}(f) &=&  \inf_{h \in C_s^+(\mathbb{R}^n), h \leq f} as_{\lam} ^{(s)}(h) \leq  
 \inf_{\alpha \,  g^{(s)}_e \leq f} as_{\lam} ^{(s)}(\alpha \, g^{(s)}_e) 
 = \inf_{\alpha \,  g^{(s)}_e \leq f}  \alpha ^{1-2 \lam} \int_{\mathbb{R}^n} g^{(s)}_e dx\\
&= &0.
\end{eqnarray*}
By the duality relation Theorem \ref{basic} (iv), and as $(f_{(s)}^\circ)_{(s)}^\circ = f$, this is equivalent to $os_\lam^{(s)}(f)=0$ for $\lam \in (\frac{1}{2}, \infty)$.
\par
\noindent
We now show that 
\begin{equation}\label{isgrosser1/2}
is_\lam^{(s)}(f)=0, \hskip 10mm \frac{1}{2} < \lam < 1.
\end{equation}
By the duality relation Theorem \ref{basic} (iv), this is equivalent to showing that $os_\lam^{(s)}(f)=0$ for 
$\lam \in (0, \frac{1}{2})$.
\newline
For $R, b \in \mathbb{R_+}$, let $h^s_{R,b}(x) = b- R \, |x|^2$. We consider  only these $R, b$ such that $h_{R,b}(x)^s \geq f^s(x)$. Let first $s \leq 1$. Then
\begin{eqnarray} \label{Faelle}
0 \leq os_{\lam}^{(s)}(f) &\leq &os_{\lam}^{(1)}(f) =  \inf_{h \geq f, h \in {\mathit C}^+_{1}(\R^n)} as_{\lam}^{(1)} (h) \leq  \inf_{ h_{R,b}  \geq f} as_{\lam} ^{(1)}(h_{R,b})\nonumber \\
&=& \inf_{ h_{R,b}  \geq f}   \int_{\left(\frac{b}{R}\right)^\frac{1}{2}B^n_2} \, \frac{ (2\,R) ^{\lam n}}{\left( b+R \, |x|^2 \right)^{\lam(n+2)-1}}dx.
\end{eqnarray}
We  first treat the case that $\lam(n+2)-1>0$.  Then  $b+R \, |x|^2 \geq b$ for all $x$, and thus (\ref{Faelle}) becomes
\begin{eqnarray*} 
0 \leq os_{\lam}^{(s)}(f) &\leq&  \inf_{R,b}\,  2^{\lam n}  |B^n_2| \left(\frac{b}{R}\right)^{n (\frac{1}{2} -\lam)} b^{2(\frac{1}{2}-\lam)}.
\end{eqnarray*}
We then choose $R\geq b^5$ and the last inequality can  be  estimated  by
\begin{eqnarray*} 
0 \leq os_{\lam}^{(s)}(f) &\leq&  \inf_{b} \, |B^n_2|  \frac{2^{\lam n}}{b^{2(\frac{1}{2}-\lam)(2n-1)}} =0, 
\end{eqnarray*}
as $0< \lam  <\frac{1}{2}$.
Next we look at the case $\lam(n+2)-1 \leq 0$. Then we use  that $|x|^2 \leq \frac{b}{R}$ and get
\begin{eqnarray*} 
0 \leq os_{\lam}^{(s)}(f) &\leq &os_{\lam}^{(1)}(f) \leq  \inf_{ h_{R,b}  \geq f} \left(2 \, R \right)^{\lam n}  \int_{\left(\frac{b}{R}\right)^\frac{1}{2}B^n_2} \, \left( b+R \, |x|^2 \right)^{1-\lam(n+2)}dx\\
& \leq & \inf_{ R,b} \left(2 \, R \right)^{\lam n} |B^n_2| \left(\frac{b}{R}\right)^\frac{n}{2} \left(2b\right)^{1-\lam(n+2)} = 2^{1-2 \lam} |B^n_2| \inf_{ R,b}  \left(\frac{b^{n+2}}{R^n}\right)^\frac{1-2 \lam}{2}.
\end{eqnarray*}
If $\lam(n+2)-1 = 0$, we can choose $b=R^\frac{1}{2}$ and if $\lam(n+2)-1 < 0$, then we choose $b=R^\frac{n-1}{n+2}$ and get in both cases 
 that $os_{\lam}^{(s)}(f)=0$, as $0< \lam  <\frac{1}{2}$.
\par
\noindent
Now we consider the case when $s >1$. 
\begin{eqnarray} \label{WasSollDas}
os_{\lam}^{(s)}(f)  &\leq&  \inf_{ h_{R,b}  \geq f} as_{\lam} ^{(s)}(h_{R,b}) \nonumber\\
&=& \left(\frac{2}{s} \right)^{\lam n} \frac{1}{ns+1} \inf_{ h_{R,b}  \geq f} R^{\lam n}  
\int_{\left(\frac{b}{R}\right)^\frac{1}{2}B^n_2} \, \frac{\left( b-R \, |x|^2 \right)^{(1-\lam)(\frac{1}{s}-1)}}{\left( b+R \, |x|^2 \right)^{\lam(n+\frac{1}{s}+1)-1}}dx\nonumber \\
&=& \left(\frac{2}{s} \right)^{\lam n} \frac{1}{ns+1} \inf_{ R,b} \left[\left( \frac{b}{R} \right)^{ n}  b^\frac{2}{s}\right]^{\frac{1}{2} -\lam}
\int_{B^n_2} \, \frac{\left( 1- |x|^2 \right)^{(1-\lam)(\frac{1}{s}-1)}}{\left( 1+ \, |x|^2 \right)^{\lam(n+\frac{1}{s}+1)-1}}dx.
\end{eqnarray}
If $\lam(n+\frac{1}{s}+1)-1 \geq 0$, we estimate $1+ \, |x|^2 \geq 1$ and if $\lam(n+\frac{1}{s}+1)-1 < 0$, we estimate $1+ \, |x|^2 \leq 2$. In both cases we get
with an absolute constant $C$ that (\ref{WasSollDas}) is
$$
\leq C  \left(\frac{2}{s} \right)^{\lam n} \frac{|S^{n-1}| }{ns+1} B\left(\frac{1}{s}(1-\lam) +\lam, \frac{n}{2}\right)  \inf_{ R,b} \left[\left( \frac{b}{R} \right)^{ n}  b^\frac{2}{s}\right]^{\frac{1}{2} -\lam}, 
$$
where $B(x, y) = \int_0^1 t^{x-1} (1-t)^{y-1} dt$ is the Beta function. We choose e.g., $R=b^3$ and get that the latter expression is equal to $0$.
\par
\noindent
It remains to settle the case $\lam = 1/2$. We treat that together with  $is_{\frac{1}{2}}^{(s)}$ and will get that $0=is_{\frac{1}{2}}^{(s)}(f)= os_{\frac{1}{2}}^{(s)}(f_{(s)}^\circ)$.
By Proposition \ref{mono-asa} we have for $\mu > \frac{1}{2}$, 
\begin{eqnarray*}
is_{\frac{1}{2}}^{(s)}(f)&=&\inf_{h \leq f, h \in {\mathit C}^+_{s}(\R^n)} as_{\frac{1}{2}}^{(s)}(h) \leq \inf_{h \leq f, h \in {\mathit C}^+_{s}(\R^n)}\left(as_{\mu}^{(s)}(h)\right)^\frac{1}{2 \mu} \left(\int_{\mathbb{R}^n} h\right)^{1-\frac{1}{2 \mu}}\\
&\leq &\left(\int_{\mathbb{R}^n} f \right)^{1-\frac{1}{2 \mu}} \left(is_{\mu}^{(s)}(f)\right)^\frac{1}{2 \mu} =0.
\end{eqnarray*}
The last equality holds by (\ref{isgrosser1/2}).
\par
\noindent
Finally, similarly as for $OS_\lam^{(s)}$ one shows that $0 < os_\lam^{(s)}(f) < \infty $ on $(- \infty, 0)$. 
\vskip 2mm
\noindent
{\bf Conclusion}. The relevant $\lam$-range for $os_\lam^{(s)}$ is $\lam \in (- \infty, 0]$ and the relevant range  for the inner minimal affine surface area $is_\lam^{(s)}$ is $[1, \infty)$.
\vskip 3mm
\noindent
Now we show that the extrema are attained when $\frac{1}{s} \in \mathbb{N}$.
We treat the $IS_\lam^{(s)}$ case. The others are done accordingly. Let $\lam \in [0,1/2]$. Let $p=\left(n+\frac{1}{s}\right) \frac {\lam}{1-\lam}$.
\begin{eqnarray*}
 IS_\lam^{(s)}(f) &=&  \sup_{h \in C_s^+(\mathbb{R}^n), h \leq f} as_{\lam} ^{(s)}(h) =  \sup_{K_s(h) \subset K_s(f)} \frac{as_{p}\left(K_s(h)\right)}{s^\frac{n}{2}(n+\frac{1}{s}) |B^\frac{1}{s}_2|} 
 = \frac{as_{p}\left(K_s(h_0)\right)}{s^\frac{n}{2} (n+\frac{1}{s}) |B^\frac{1}{s}_2|} \\
 &=&  as_{\lam} ^{(s)}(h_0).
\end{eqnarray*} 
The supremum is attained by upper semi-continuity of the $p$-affine surface areas for bodies \cite{Lutwak96}, see also \cite{Ludwig2010}.
\par
\noindent
This finishes the proof of Proposition \ref{ranges}.
\end{proof}
\vskip 4mm
\noindent

\subsection{Proof of Theorem \ref{basic}.}

\begin{proof}
\par
\noindent
(i) The invariance properties follow from Proposition \ref{properties}.
\par
\noindent
(ii) As ${\mathit C}_{s_2}^+(\R^n) \subset {\mathit C}_{s_1}^+(\R^n)$, $IS_\lam^{(s_1)} (f)$ is defined and
\begin{eqnarray*}
IS_\lam^{(s_1)} (f)=\sup_{h \leq f,  h\in {\mathit C}_{s_1}^+(\R^n)} as_\lam^{(s_1)}(h) \geq \sup_{h \leq f, h \in {\mathit C}_{s_2}^+(\R^n)} as_\lam^{(s_2)}(h)=
IS_\lam^{(s_2)} (f).
\end{eqnarray*} 
The arguments for $OS^{(s)}_\lam$, $os^{(s)}_\lam$ and $is^{(s)}_\lam$ are  the same.
\par
\noindent
(iii) For $\lam \in [0, \frac{1}{2}]$,  we have by Proposition \ref{Ungleichung1}, 
\begin{eqnarray*}
IS_{\lam}^{(s)}(f) &=&  \sup_{h \leq f, h \in {\mathit C}^+_{s}(\R^n) } as_{\lam} ^{(s)}(h) \\
&\leq& \sup_{h \leq f, h \in {\mathit C}^+_{s}(\R^n)} \left(\int_{\R^n}  h\, dx \right)^{1-\lam}
 \left( \int_{\R^n}  h_{(s)}^\circ  \  dx \right)^{ \lam}.
\end{eqnarray*}
Now we apply,  as above,  the general form of the functional Blaschke-Santal\'o inequality (\ref{concsantalo}), 
again with $f=h$, $g=h_{(s)}^\circ$ and $\rho(t) = (1-s t)^\frac{1}{2s}$
and  get as above
$$
\int_{\R^n} h \, dx  \int_{\R^n} h_{(s)}^\circ \, dx   \leq \left(\int_{\R^n}   \left(1- s |x|^2\right)^\frac{1}{2 s}\, dx\right)^2.
$$  
Hence
\begin{eqnarray*}
IS_{\lam}^{(s)}(f) &\leq&  \sup_{h \leq f, h \in {\mathit C}^+_{s}(\R^n)}  \left(\int_{\R^n}  h \,  dx \right)^{1-2 \lam} \left(\int_{\R^n}   \left(1- s |x|^2\right)^\frac{1}{2 s}\, dx\right)^{2 \lam}\\
&\leq&  \left(\int_{\R^n}  f \,  dx \right)^{1-2 \lam} \left(\int_{\R^n}  g^{(s)} _e dx\right)^{2 \lam}.
\end{eqnarray*}
As $IS_\lam^{(s)}(g^{(s)}_e)= \int_{\mathbb{R}^n} g^{(s)}_e dx$ by Lemma \ref{ge-prop},  the statement  follows.
\par
\noindent
Next we address the equality cases. 
By the proof of Theorem \ref{ranges},  we have for all $f$ that $IS_{0}^{(s)}(f)= \int_{\mathbb{R}^n} f dx$, which settles the case of $\lam=0$. 
Likewise, by the proof of Theorem \ref{ranges},  we have for all $f$ that $IS_{\frac{1}{2}}^{(s)}(f)= \int_{\mathbb{R}^n} g^{(s)}_e dx$, which settles the case of $\lam=\frac{1}{2}$. 
\par
\noindent
The proofs for  $OS_\lam^{(s)}$, $os_\lam^{(s)}$ and $is^{(s)}_\lam$ are accordingly. 
\vskip 2mm
\noindent
(iv)  Let $\lam \in [0, 1/2]$. By Proposition \ref{properties}, 
\begin{eqnarray*}
IS_{\lam}^{(s)}(f) &=&  \sup_{h \leq f, h \in {\mathit C}^+_{s}(\R^n)} as^{(s)}_{\lam}(h) =  \sup_{h \leq f, h \in {\mathit C}^+_{s}(\R^n)} as^{(s)}_{1-\lam}(h_{(s)}^\circ)\\
& =&
\sup_{h_{(s)}^\circ \geq f_{(s)}^\circ, h_{(s)}^\circ \in {\mathit C}^+_{s}(\R^n)} as^{(s)}_{1-\lam}(h_{(s)}^\circ) = OS_{1-\lam}^{(s)}(f_{(s)}^\circ),
\end{eqnarray*}
and similarly for the other quantities.
\vskip 2mm
\noindent
(v) Let $\mu \geq \lam >0$. Then, by Proposition \ref{mono-asa}, 
\begin{eqnarray*}
IS_{\lam}^{(s)}(f) &=&  \sup_{h \leq f, h \in {\mathit C}^+_{s}(\R^n)} as^{(s)}_{\lam}(h) \leq \sup_{h \leq f, h \in {\mathit C}^+_{s}(\R^n)}\left( as^{(s)}_{\mu}(h)\right)^\frac{\lam}{\mu} 
\left(\int _{\mathbb{R}^n} h\right)^{1-\frac{\lam}{\mu}}\\
&\leq & \left(\int _{\mathbb{R}^n} f \right)^{1-\frac{\lam}{\mu}}  \left(IS_{\mu}^{(s)}(f) \right)^\frac{\lam}{\mu}, 
\end{eqnarray*}
and from this the claim follows.
Let $\nu \geq \rho$, $\nu \in [\frac{1}{2}, 1]$. Let $\lam = 1-\nu$ and $\mu = 1-\rho$. Then $\mu \geq \lam$ and by the monotonicity in $\lam$ of $IS_{\lam}^{(s)}(f) $ and the duality relation (iv), 
\begin{eqnarray*}
OS_{\nu}^{(s)}(f) &=& IS_{\lam}^{(s)}(f_{(s)}^\circ) \leq \left(\int_{\mathbb{R}^n} f_{(s)}^\circ\right)^{1-\frac{\lam}{\mu}} \left(IS_{\mu}^{(s)}(f_{(s)}^\circ)\right)^\frac{\lam}{\mu} \\
&=& \left(\int_{\mathbb{R}^n} f_{(s)}^\circ \right)^{1-\frac{\lam}{\mu}} \left(OS_{\rho}^{(s)}(f) \right)^\frac{\lam}{\mu},
\end{eqnarray*}
and the claim follows.
\par
\noindent
The statements for $is_{\lam}^{(s)}(f)$ and $os_{\lam}^{(s)}(f)$ are proved similarly.
\end{proof}

\vskip 4mm
\noindent
\subsection{Proof of Theorem \ref{prelprop} and  Proposition \ref{prelprop1}.} \label{subsection5.3}

For the proof of these theorems we need another lemma. We  recall the  notion of isotropy constant.
The isotropic constant 
$L_K^2$ of a convex body  $K$ in $\mathbb{R}^n$, see, e.g., \cite{ArtsteinGiannopoulosMilmanBook, BGVV14}, is  defined by 
\begin{equation} \label{LK}
n L_K^2= \min \left\{\frac{1}{\mbox{vol}_{n}\left(TK\right) ^{1 + \frac{2}{n}}} \int_{a+TK} |x | ^2 dx : a \in \mathbb{R}^n, T \in GL(n) \right\}.
\end{equation}
\vskip 2mm
\noindent
We will also need the isotropic constant of an $s$-concave function $f$ on $\mathbb{R}^n$. 
If $f$ has  finite, positive integral,   its inertia- or covariance-matrix $Cov(f)$ is  the matrix with entries
\begin{equation}\label{Cov}
\left[Cov(f)\right]_{ i,j }=\frac{\int_{\mathbb{R}^n} x_i x_j f(x) dx}{\int_{\mathbb{R}^n} f (x)dx} - 
\frac{\int_{\mathbb{R}^n} x_i f(x) dx}{\int_{\mathbb{R}^n} f (x)dx} \, \frac{\int_{\mathbb{R}^n}  x_j f(x) dx}{\int_{\mathbb{R}^n} f(x)dx}, 
\end{equation}
see, e.g., \cite{ArtsteinGiannopoulosMilmanBook, BGVV14}.
Then the isotropic constant $L_f$ is defined as (e.g.,  \cite{ArtsteinGiannopoulosMilmanBook, BGVV14}), 
\begin{equation}\label{Lphi}
L_f = \left(\frac{\|f\|_{\infty}}{\int_{\mathbb{R}^n} f (x)dx} \right)^\frac{1}{n}\, \left( \det \, Cov(f)\right)^\frac{1}{2n}, 
\end{equation}
where $\|f\|_{\infty} = \sup_{x \in \mathbb{R}^n} f(x)$.
\par
\noindent
We say that $f$  isotropic if  $\int_{\mathbb{R}^n} f dx = \alpha$ for some constant $\alpha>0$, its barycenter is at the origin and $\int_{\mathbb{R}^n} x_ix_j f dx = \delta_{ij}$ for all $i,j =1, \dots, n$. Every $s$-concave function  $f$ with finite positive integral has an isotropic image: there exists an affine  isomorphism
$T : \mathbb{R}^n \to \mathbb{R}^n$ and a positive number $a$ such that $a f \circ T$ is isotropic.
We also remark that if $f$ is an $s$-concave isotropic function, then
\begin{equation}\label{norm}
\|f\|_{\infty}^\frac{1}{n} \geq  c,
\end{equation}
where $c$ is an absolute constant.
\vskip 2mm
\noindent
{\bf Remarks.}
\par
\noindent
1. Note  that if $K$ is a convex body, then $L_K=L_{\mathbbm{1}_K}$.  
\vskip 2mm
\noindent
2. Similar to the isotropy constant conjecture for convex bodies, one asks the isotropy constant conjecture for functions:
there is an absolute constant $C>0$ such that for any isotropic log-concave (the case $s=0$) density $f : \mathbb{R}^n \to [0,\infty)$ we have that
$\|f\|_{\infty}^\frac{1}{n} \leq C$ and similar observations hold as in the case of bodies.
\vskip 3mm
\noindent
We want to compare the isotropic constants of an $s$-concave function $f$ with that of its associated body  $K_s(f)$.
The next lemma states how $L_f$ and $L_{K_s(f)}$ are related. We could not find a proof for this and for completeness give one in an appendix.
\vskip 3mm
\noindent
\begin{lem}\label{lemLphi}
 Let $s>0$ be such that $\frac{1}{s} \in \mathbb{N}$. Let $ f \in {\mathit C}_s(\R^n)$.  Then
\par
(i) \, \, $L_{K_s(f)}^{n+\frac{1}{s}}= L_f^n \, \left(\frac{\left( \int_{\mathbb{R}^n} f^{2s+1} dx \right)^\frac{1}{2s}}{2^\frac{1}{2s}\left(1+\frac{1}{2s}\right)^{\frac{1}{2s}}  \|f\|_\infty \left( \int_{\mathbb{R}^n} f dx \right)^\frac{1}{2s}\text{vol}_\frac{1}{s}\left(B_2^\frac{1}{s}\right)}\right).$
\par
(ii)\, \,  $L_{K_s(f)^\circ}^{n+\frac{1}{s}}= L_{f^\circ_{(s)}}^n \, \left(\frac{\left( \int_{\mathbb{R}^n}\left(f_{(s)}^\circ\right)^{2s+1} dx \right)^\frac{1}{2s}}{2^\frac{1}{2s}\left(1+\frac{1}{2s}\right)^{\frac{1}{2s}}  \|f_{(s)}^\circ \|_\infty \left( \int_{\mathbb{R}^n} f ^\circ_{(s)}dx \right)^\frac{1}{2s}\text{vol}_\frac{1}{s}\left(B_2^\frac{1}{s}\right)}\right).$
\end{lem}
\vskip 3mm
\noindent
In part of the proof below it  is most convenient to work with the convex bodies  in isotropic position. A convex body $K\subseteq \R^n$ is said to be in isotropic position if $|K| =1 $, $K$ is centered, i.e., has its barycenter at $0$  and 
there is $a > 0$ such that 
\begin{align*}
\int_K x_i x_j  dx = a^2 \delta_{ij}, i, j=1, \dots, n. 
\end{align*}
It is known that for every convex body $K\subseteq \R^n$, there exists $T:\R^n \to \R^n$ affine and invertible such that $TK$ is isotropic. We refer again to  e.g.  \cite{ArtsteinGiannopoulosMilmanBook, BGVV14} for the details  and other facts on isotropic position used here.
\vskip 2mm
\noindent

\subsubsection{Proof of Theorem \ref{prelprop}.}

\begin{proof} 
(i) The estimate from above and the equality cases are Theorem \ref{basic}, (iii).
\par
\noindent
The first step to establish the lower bound uses a construction  of \cite{GiladiHuangSchuettWerner}. 
We pass from functions to convex bodies via $K_s(f) \subset  \mathbb{R}^{n+\frac{1}{s}}$.   But note that the theorem does not follow from \cite{GiladiHuangSchuettWerner}, as our definitions for maximal affine surface areas are different from \cite{GiladiHuangSchuettWerner}.
 As the expressions  that we estimate  are invariant under linear transformations, we can assume that  $f$ is in isotropic position, which implies that 
 $K_s(f)$ is  in isotropic position.
In particular,  this means that we can assume that $1=|K_s(f)|=s^\frac{n}{2} |B^\frac{1}{s}| \int_{\mathbb{R}^n} f\, dx$. 
We put $n_s=n+\frac{1}{s}$. 
Let $L_{K_s(f)}$ be the isotropic constant of $K_s(f)$. By the thin shell estimate of O. Gu\'edon and E. Milman \cite{GuedonMilman} (see also \cite{FleuryGuedonPaouris, LeeVempala17, Paouris}), we have with  universal constants $c$ and $C$, that for all $t \geq 0$, 
\begin{eqnarray*} 
\left| K_s(f) \cap \left\{ x\in \mathbb{R}^{n_s} :\, \left| |x|- L_{K_s(f)} \sqrt{n_s} \right| <  t \, L_{K_s(f)} \sqrt{n_s} \right\} \right| \\
>1 -  C \text{exp}(-c(n_s)^{1/2}\text{min}(t^3,\, t)).
\end{eqnarray*}
Taking $t= O((n_s)^{-1/6})$, there is a  a new universal constant $c>0$ such that for all $n_s \in \mathbb{N}$, 
\begin{equation}
\left| K_s(f) \cap \left\{ x\in \mathbb{R}^{n_s}: \, \left| |x|-\sqrt{n_s} L_{K_s(f)}\right| < c \, L_{K_s(f)}\,  n_s^\frac{1}{3} \right\}\right| \geq \frac{1}{2}.  \label{thin shell}
\end{equation}
This  set consists of all  $x \in K_s(f)$  for which
 $$
 L_{K_s(f)} \left( (n_s)^{1/2} - c \, n_s^{1/3}\right) <  |x|  <  L_{K_s(f)} \left( n_s^{1/2} +  c \, n_s^\frac{1}{3} \right).
 $$ 
We consider those $n_s \in \mathbb{N}$ for which 
$n_s^{1/6} >c$.
\newline
We will truncate the above set. For $i=0,\,1,\,2,\,\dots k_{n_s} =  \lfloor{n_s \log_2 \frac {n_s^\frac{1}{2} + c \, n_s^\frac{1}{3} }{(n_s)^{1/2}- c \, n_s^{1/3}}} \rfloor$, consider the sets 
$$
    L_i := K_s(f)\cap \{x\in \mathbb{R}^{n_s} \, :\, 2^{i/n_s} (L_{K_s(f)} ((n_s)^{1/2}-c n_s^{1/3}) )< |x| \le 2^{(i+1)/n_s}(L_{K_s(f)} (n_s^{1/2}-cn_s^{1/3}))\}.
$$
Then
$$
2^{\frac{k_{n_s}}{n_s }} \leq 2^{\log_2 \frac {n_s^{1/2}+ c n_s^{1/3}}{n_s^{1/2}- c \, n_s^{1/3}}} = \frac {n_s^{1/2}+ c\,  n_s^{1/3}}{n_s^{1/2}- c\,  n_s^{1/3}}
$$
and thus
\begin{align}
    K_s(f) \cap \left\{ x\in \mathbb{R}^{n_s}\,  :\left| |x|-L_{K_s(f)}\sqrt{n_s} \right| < c \, L_{K_s(f)} n_s^{1/3} \right\} \subset \cup_{i=0}^{k_{n_s}} L_i \label{truncation}.
\end{align}
Moreover, with a new absolute constant $C_0$, 
$$
k_{n_s} \leq n_s \log_2 \frac {n_s^{1/2}+ c n_s^{1/3}}{n_s^{1/2}- c \, n_s^{1/3}} = n_s  \log_2 \frac {1 + c \, n_s^{-1/6}}{ 1 - c \, n_s^{-1/6}} \leq  \ C_0\,  n_s^{5/6}.
$$
By (\ref{thin shell}) and (\ref{truncation}), there exists $i_0 \in \{1,\,2,\,\dots,\, \lfloor {C_0 n_s^{5/6}}\rfloor \}$ such that 
\begin{equation}
|L_{i_0} |\ge \frac{1}{2 \   \lfloor {C_0\,  n_s^{5/6}}\rfloor}. \label{eq:Livolume}
\end{equation}
We set 
\begin{equation}\label{GrossR}
R=2^{i_0/n_s}(L_{K_s(f)}( n_s^{1/2}-c n_s^{1/3})). 
\end{equation}
In particular, we have 
$$
L_{i_0}= K_s(f)\cap \{x\in \mathbb{R}^{n_s} \, :\, R < |x| \le 2^{1/n_s}R\}.
$$ 
Let 
\[
O=\left\{ \theta\in S^{n_s-1}\,:\, \rho_{ K_s(f)}\left(\theta\right)>R \right\} ,\,\text{and}\,\,S_{O}=\left\{ r\, \theta\,:\,\theta\in O\,{\rm and}\,r\in\left[0,\,R\right]\right\} \subset K_s(f), 
\]
where $\rho_{K_s(f)}\left(\theta\right)=\max\left\{ r\ge0\,:\,r\theta\in K_s(f)\right\} $
is the radial function of $K_s(f)$.
\par
\noindent
Now we claim that 
\begin{equation}
L_{i_0} \subset 2^{1/n_s}S_O.
\label{eq:inclusion}
\end{equation}
Indeed, let  $y\in L_{i_0}$. We express $y=r\, \theta$ in polar coordinates. By definition, we have $R<r<2^{1/n_s}R$ and $r\, \theta \in K_s(f)$. Thus, $\rho_{K_s(f)}(\theta)\ge r>R$ and hence $\theta\in O$. Therefore, $r\theta\in 2^{1/n_s}S_O$ because $r\in [0,\,2^{1/n_s}R]$.
By (\ref{eq:Livolume}) and (\ref{eq:inclusion}) we conclude that
\begin{equation}
\left|S_{O}\right|\ge\left(2^{-1/n_s}\right)^{n_s}|L_{i_0}| \ge \frac{1}{4 \   \lfloor {C_0 n_s^{5/6}}\rfloor}.\label{eq:volume}
\end{equation}
Let $h$ be this  function such that $K_s(h) = K_s(f) \cap R\, B_{2}^{n_s}$. Then $h \leq f$,  $h$ is in $C_s(\mathbb{R}^n)$,  but not necessarily in $C_s^+(\mathbb{R}^n)$, apart from 
the part of the function  $h$ that corresponds to $R\, B_{2}^{n_s}$ which is in $C_s^+(\mathbb{R}^n)$.  We   construct a new $h_0$ such that $h_0 \in C_s^+(\mathbb{R}^n)$ such that the part corresponding to  $R\, B_{2}^{n_s}$ coincides with $h$.  This can be done 
by convolving $h$ with a mollifier, see e.g., \cite{Evans}, Theorem 7. 
Then $h_0 \in C_s^+(\mathbb{R}^n)$ and still $h_0 \leq f$. 
With (\ref{motiv1}),  and $p= n_s\, \frac{\lam}{1-\lam}$, 
\begin{eqnarray}\label{below1}
IS_\lam^{(s)} (f) & \geq& as_{\lam} ^{(s)}(h_0) = \frac{as_{p} \left(K_s(h_0) \right)}{ (n +\frac{1}{s})\,  s^\frac{n}{2} \, \text{vol}_{\frac{1}{s}} \left(B_2^{\frac{1}{s}}\right) }.
\end{eqnarray}
We now estimate $as_{p}\left(K_s(h_0)\right)$.
For $\theta\in O$, $R\,  \theta$ is a boundary
point of $K_s(h_0)$. Thus,
\begin{eqnarray*}
as_{p}\left(K_s(h_0)\right) & \ge\int_{R\, O}\frac{\kappa^{\frac{p}{n_s+p}}}{\langle x, N\left(x\right)\rangle ^{\frac{n_s\left(p-1\right)}{n_s+p}}}d{\mu\left(x\right)}
  =\int_{R\, O}\frac{R^{-\left(n_s-1\right)\frac{p}{n_s+p}}}{R^{\frac{n_s\left(p-1\right)}{n_s+p}}}d{\mu\left(x\right)}\\
 & =\mu\left(R\, O\right)\left(\frac{1}{R}\right)^{\frac{\left(n_s-1\right)p+n_s\left(p-1\right)}{n_s+p}}
   =\mu\left(R\, O\right)\left(\frac{1}{R}\right)^{\frac{2n_sp}{n_s+p}-1},
\end{eqnarray*}
where $\mu$ is the surface area measure of $R\, S^{n_s-1}$.
We can compare surface area and volume, 
\begin{equation}\label{Surf-Vol}
\frac{\mu\left(R\, O\right)\cdot R}{n_s}=\left|S_{O}\right|.
\end{equation}
Hence by (\ref{eq:volume}), 
\begin{align*}
as_{p}\left(K_s(h_0\right) & \ge\left(\frac{1}{R}\right)^{\frac{2n_s p}{n_s+p}-1}\frac{n_s}{R}\left|S_{O}\right|=\left(\frac{1}{R}\right)^{\frac{2n_sp}{n+p}}n_s \left|S_{O}\right|\\
 & \ge\left(\frac{1}{R}\right)^{\frac{2n_sp}{n_s+p}} \frac{n_s}{4 \   \lfloor {C_0 n_s^{5/6}}\rfloor} \geq  \frac{n_s}{4 \   \lfloor {C_0 n_s^{5/6}}\rfloor} \left( \frac{1}{2\sqrt{n_s} L_{K_s(f)}}\right)^{\frac{2n_sp}{n_s+p}} .
\end{align*}
The last inequality follows since  $R\le 2\sqrt{n_s} L_{K_s(f)}$ by (\ref{GrossR}).
Using (\ref{below1}),  Lemma \ref{lemLphi} and the fact that $\lam=\frac{p}{n_s+p}$, 
\begin{eqnarray*}\label{below2}
IS_\lam^{(s)} (f) & \geq&  \frac{1}{ n_s\,  s^\frac{n}{2} \,  \left|B_2^{\frac{1}{s}}\right| }  \frac{n_s}{4 \   \lfloor {C_0 n_s^{5/6}}\rfloor} \frac{2^{-2 n_s \lam}}{ n_s^{ \lam n_s} L_f^{2n \lam}}
\left(  \frac{
2^\frac{1}{2s}\left(1+\frac{1}{2s}\right)^{\frac{1}{2s}} \|f\|_\infty \left(  \int_{\mathbb{R}^n} f dx \right)^\frac{1}{2s} |B_2^\frac{1}{s}|
}{\left( \int_{\mathbb{R}^n} f^{2s+1} dx \right)^\frac{1}{2s}}
\right)^{2\lam}.  
\end{eqnarray*}
Now we use that by (\ref{norm}), $\|f\|_\infty  \geq c^n$
and we use that 
\begin{eqnarray} \label{Est1}
2^\frac{\lambda}{s} \left( 1+ \frac{1}{2s}\right)^\frac{\lambda}{s}\text{vol}_{\frac{1}{s}}\left(B_2^{\frac{1}{s}}\right)^{2\lambda} \geq 2^\frac{\lambda}{s} \left(\frac{1}{2s}\right)^\frac{\lambda}{s}\text{vol}_{\frac{1}{s}}\left(B_2^{\frac{1}{s}}\right)^{2\lambda} 
\geq \left(\frac{s}{\pi}\right)^\lam\left(2 \pi  e\right)^\frac{\lam}{s}.
\end{eqnarray} 
As for $\alpha >0$, $\frac{IS_\lam^{(s)} (\alpha f) }{\left(\int_{\mathbb{R}^n} \alpha  f \, dx \right) ^{1-2 \lam}} =\frac{IS_\lam^{(s)} (f) }{\left(\int_{\mathbb{R}^n}  f \, dx \right) ^{1-2 \lam}} $
 we can assume that  $0 \leq f(x) \leq 1$. If not, replace $f$ by $\frac{f}{\|f\|_\infty}$. 
Hence $\frac{ \int_{\mathbb{R}^n} f  }{ \int_{\mathbb{R}^n} f^{2s+1} } 
\geq 1$. We get 
\begin{eqnarray*}\label{below2}
IS_\lam^{(s)} (f) & \geq&  \frac{1}{   s^\frac{n}{2} \,  \left|B_2^{\frac{1}{s}}\right| }  \frac{c^n}{4 \   \lfloor {C_0 n_s^{5/6}}\rfloor} \frac{2^{-2 n_s \lam}}{n_s^{\lam n_s} L_f^{2n \lam}}
\left(\frac{s}{\pi}\right)^\lam\left(2 \pi  e\right)^\frac{\lam}{s}.
\end{eqnarray*}
As $K_s(f)$ is in isotropic position,   $1= s^\frac{n}{2} |B^\frac{1}{s}| \int_{\mathbb{R}^n} f\, dx$.  Thus 
\begin{eqnarray*}\label{below2}
\frac{IS_\lam^{(s)} (f)}{\left(\int_{\mathbb{R}^n} f dx \right)^{1-2 \lam}} & \geq& \frac{1}{ \left(  s^\frac{n}{2} \,  \left|B_2^{\frac{1}{s}}\right| \right)^{2 \lam}}  
  \frac{2^{-2 n_s \lam}}{4 \   \lfloor {C_0 n_s^{5/6}}\rfloor} \frac{c^n}{ n_s^{\lam n_s} L_f^{2n  \lam} }
\left(\frac{s}{\pi}\right)^\lam\left(2 \pi  e\right)^\frac{\lam}{s}.
\end{eqnarray*}
We compare this to the corresponding expression for $g^{(s)}_e$. By Lemma \ref{ge-prop}, 
\begin{equation*}\label{iso-ball}
\frac{IS_\lam^{(s)} (g^{(s)}_e) }{\left(\int_{\mathbb{R}^n} g_e dx \right)^{1-2 \lam}}=\left(\int_{\mathbb{R}^n} g^{(s)}_e dx\right)^{2 \lam} = \left( \frac{ \left|B_2^{n_
s}\right|}{s^\frac{n}{2} \,  \left|B_2^{\frac{1}{s}}\right|}\right)^{2 \lam}.
\end{equation*}
Therefore, 
\begin{eqnarray*}\label{below2}
\frac{IS_\lam^{(s)} (f)}{\left(\int_{\mathbb{R}^n} f dx \right)^{1-2 \lam}} & \geq& \frac{IS_\lam^{(s)} (g^{(s)}_e) }{\left(\int_{\mathbb{R}^n} g^{(s)}_e dx \right)^{1-2 \lam}}
\left(\frac{1}{  \left|B_2^{n_s}\right| } \right)^{2 \lam}
  \frac{2^{-2 n_s \lam}}{4 \   \lfloor {C_0 n_s^{5/6}}\rfloor} \frac{1}{ n_s^{\lam n_s} L_f^{2 n \lam} }
\left(\frac{s}{\pi}\right)^\lam\left(2 \pi  e\right)^\frac{\lam}{s}.
\end{eqnarray*}
Using Sterling's formula, $|B_2^{n_s}| = \frac{\pi^\frac{n_s}{2}}{\Gamma(\frac{n_s}{2} +1)} \leq \left(\frac{2 \pi e}{n_s}\right)^\frac{n_s+1}{2}$, 
leads to 
\begin{eqnarray*}\label{below2}
\frac{IS_\lam^{(s)} (f)}{\left(\int_{\mathbb{R}^n} f dx \right)^{1-2 \lam}} & \geq& 
\frac{s^\lam}{ 2^{2(\lam n_s+1)}\left(2\pi e \right)^{\lam(n+2)}}
\frac{IS_\lam^{(s)} (g^{(s)}_e) }{\left(\int_{\mathbb{R}^n} g^{(s)}_e dx \right)^{1-2 \lam}}
  \frac{1}{ \lfloor {C_0 n_s^{5/6-\lam}}\rfloor} \frac{1}{ L_f^{2 n \lam} }.
\end{eqnarray*}
\vskip 3mm
\noindent
For the other lower bound we use the lower bound (\ref{2-OS}) of (ii) derived below. 
\par
\noindent
With   (\ref{2-OS}) and  the duality relation of
Theorem \ref{basic}, (iv). Thus, for $0 \leq \lam \leq \frac{1}{2}$ and with $\mu=1-\lam$,  we get, also using that 
$\int_{\mathbb{R}^n} f ^\circ_{(s)} dx \leq \frac{\left(\int_{\mathbb{R}^n} g^{(s)}_e dx \right)^2}{\int_{\mathbb{R}^n} f dx }$, 
\begin{eqnarray*}\label{below2}
\frac{IS_\lam^{(s)} (f)}{\left(\int_{\mathbb{R}^n} f dx \right)^{1-2 \lam}} 
&=&
 \frac{OS_\mu^{(s)} (f ^\circ_{(s)})\left(\int_{\mathbb{R}^n} f ^\circ_{(s)} dx \right)^{1-2 \mu}}{\left(\int_{\mathbb{R}^n} f dx \right)^{1-2 \lam}
\left(\int_{\mathbb{R}^n} f ^\circ_{(s)} dx \right)^{1-2 \mu}} \\
&\geq&  \frac{n_s^{n_s(1-2 \mu)}}{\left(\int_{\mathbb{R}^n} f dx \right)^{1-2 \lam}} \, 
\frac{\left(\int_{\mathbb{R}^n} g^{(s)}_e dx \right)^{2(1-2 \mu)}}{\left(\int_{\mathbb{R}^n} f dx \right)^{1-2 \mu}} \, 
\frac{OS_\mu^{(s)} (g^{(s)}_e)}{\left(\int_{\mathbb{R}^n} g^{(s)}_e dx \right)^{1-2 \mu}} \\
&= &\frac{1}{n_s^{n_s(1-2 \lam)}} \, \frac{OS_\mu^{(s)} (g^{(s)}_e)}{\left(\int_{\mathbb{R}^n} g^{(s)}_e dx \right)^{1-2 \lam}}
= \frac{1}{n_s^{n_s(1-2 \lam)}} \, \frac{IS_\lam^{(s)} ((g^{(s)}_e)^\circ)}{\left(\int_{\mathbb{R}^n} g^{(s)}_e dx \right)^{1-2 \lam}}.
\end{eqnarray*}
The estimate follows as $(g^{(s)}_e)^\circ = g^{(s)}_e$.
\vskip 3mm
\noindent
(ii) Now we give the proof of (ii). The estimate from above and the equality cases are Theorem \ref{basic}, (iii).
\par
\noindent
We derive a first  bound using L\"owner position. We will assume that $K_s(f)$ is in L\"owner position, i.e.,  the L\"owner ellipsoid $L(K_s(f))$, which is the ellipsoid of minimal volume containing $K_s(f)$,
is the Euclidean ball $\left(\frac{|L(K_s(f))|}{|B^n_2|}\right)^\frac{1}{n_s} B^{n_s}_2$.  We also have that 
\begin{equation}\label{loewner}
K_s(f) \subset L(K_s(f)) \subset n_s \ K_s(f),
\end{equation}
and that for a $0$-symmetric convex body $K_s(f)$, which happens if $f$ is even,
\begin{equation}\label{loewner2}
K_s(f) \subset L(K_s(f)) \subset \sqrt{n_s}  \ K_s(f).
\end{equation}
We choose $h$ such that $K_s(h)= L(K_s(f))$. Then $h \geq f$ and $h \in C_s^+(\mathbb{R}^n)$.
Thus,  using (\ref{loewner}) and  $\frac{n_s-p}{n_s+p} = 1-2 \lam$, 
\begin{eqnarray*}
OS_\lam^{(s)} (f)  & \geq& as_{\lam} ^{(s)}(h) = \frac{as_{p} \left(L(K_s(f) )\right)}{ n_s  s^\frac{n}{2} \,  \left|B_2^{\frac{1}{s}}\right| } 
=
\left(\frac {\left|L(K_s(f))\right|}{\left|B^{n_s}_2\right|} \right) ^\frac{n_s-p}{n_s+p}  \frac{ |B^{n_s}_2|}{ s^\frac{n}{2} \, \left|B_2^{\frac{1}{s}}\right|} \\
&\geq& n_s^{n_s (1-2 \lam)} \left(\frac {\left|K_s(f)\right|}{\left|B^{n_s}_2\right|} \right) ^{1-2 \lam} \frac{ |B^{n_s}_2|}{ s^\frac{n}{2} \,  \left|B_2^{\frac{1}{s}}\right|}\\
&=& n_s^{n_s (1-2 \lam)} \left(\frac {\left|K_s(f)\right|}{s^\frac{n}{2} \,  \left|B_2^{\frac{1}{s}}\right|} \right) ^{1-2 \lam} \frac{ |B^{n_s}_2|^{2 \lam}}{ \left(s^\frac{n}{2} \, \left|B_2^{\frac{1}{s}}\right|\right)^{2 \lam}}.
\end{eqnarray*}
Now we use 
\begin{equation}\label{Kugel}
\frac{ |B^{n_s}_2|^{2 \lam}}{ \left(s^\frac{n}{2} \,\left|B_2^{\frac{1}{s}}\right|\right)^{2 \lam}} = \frac{OS_\lam^{(s)} (g^{(s)}_e) }{\left(\int_{\mathbb{R}^n}  g^{(s)}_e\, dx \right) ^{1-2 \lam}}
\end{equation}
and
\begin{equation}\label{f-volumen}
\left(\frac {\left|K_s(f)\right|}{s^\frac{n}{2} \,  \left|B_2^{\frac{1}{s}}\right|} \right) ^{1-2 \lam} = \left(\int_{\mathbb{R}^n} f\, dx \right) ^{1-2 \lam}.
\end{equation}
Thus we get
\begin{eqnarray}\label{2-OS}
\frac{OS_\lam^{(s)} (f) }{\left(\int_{\mathbb{R}^n}  f \, dx \right) ^{1-2 \lam}} \geq  n_s^{n_s (1-2 \lam)} \frac{OS_\lam^{(s)} (g^{(s)}_e) }{\left(\int_{\mathbb{R}^n}  g^{(s)}_e\, dx \right) ^{1-2 \lam}}.
\end{eqnarray}
\vskip 2mm
\noindent
We get the other  estimate for the lower bound using the estimate for $IS$ and the duality relation of Theorem \ref{basic}, (iv). We put again $\mu=1-\lam$ and get with
part (i) of the theorem that 
\begin{eqnarray*}
\frac{OS_\lam^{(s)} (f)}{\left(\int_{\mathbb{R}^n} f dx \right)^{1-2 \lam}}   &= &
\frac{IS_\mu^{(s)} (f ^\circ_{(s)})}{
\left(\int_{\mathbb{R}^n} f ^\circ_{(s)} dx \right)^{1-2 \mu}}   
\, \frac{\left(\int_{\mathbb{R}^n} f ^\circ_{(s)} dx \right)^{1-2 \mu}}
{ \left(\int_{\mathbb{R}^n} f dx \right)^{1-2 \lam}} \\
&\geq&\frac{IS_\mu^{(s)} (g^{(s)}_e)}{\left(\int_{\mathbb{R}^n} g^{(s)}_e dx \right)^{1-2 \mu}} 
  \frac{s^\mu \, C^{n_s \mu}}{ \lfloor {C_0 n_s^{5/6-\mu}}\rfloor} \frac{1}{ L_{f_{(s)}^\circ}^{2n \mu} } \, \frac{\left(\int_{\mathbb{R}^n} f ^\circ_{(s)} dx \right)^{1-2 \mu}}
{ \left(\int_{\mathbb{R}^n} f dx \right)^{1-2 \lam}} \\
&=& \frac{OS_\mu^{(s)} (g^{(s)}_e)}{\left(\int_{\mathbb{R}^n} g^{(s)}_e dx \right)^{1-2 \lam}} 
  \frac{s^\mu \, C^{n_s \mu}}{ \lfloor {C_0 n_s^{5/6-\mu}}\rfloor} \frac{1}{ L_{f_{(s)}^\circ}^{2n \mu} }\, \left(\int g^{(s)}_e dx\right)^{2(\mu-\lam)}  \frac{\left(\int_{\mathbb{R}^n} f ^\circ_{(s)} dx \right)^{1-2 \mu}}
{ \left(\int_{\mathbb{R}^n} f dx \right)^{1-2 \lam}}.
\end{eqnarray*}
We also have used that  $(g^{(s)}_e)^\circ = g^{(s)}_e$.
As  $K_s(f_{(s)}^\circ)= (K_s(f))^\circ$ by Lemma  3.1 in \cite{ArtKlarMil}, we get 
\begin{eqnarray*}
\int_{\mathbb{R}^n} f ^\circ_{(s)} dx &=&\frac{\left|  K_s(f_{(s)}^\circ) \right|}{ n_s  s^\frac{n}{2} \, \left|B_2^{\frac{1}{s}}\right| } =\frac{\left| \left(K_s(f) \right)^\circ\right|}{ n_s  s^\frac{n}{2} \, \left|B_2^{\frac{1}{s}}\right| } \geq 
\frac{c^{n_s} \left|B_2^{n_{s}}\right|^2}{ n_s\,  s^\frac{n}{2} \, \left|B_2^{\frac{1}{s}}\right| \left|K_s(f)\right|} 
= \frac{c^{n_s} \left| B_2^{n_{s}}\right|^2}{ n_s^2\,  s^n \, \left|B_2^{\frac{1}{s}}\right|^2 \int _{\mathbb{R}^n} f dx}.
\end{eqnarray*}
In the above, we have used the reverse Blaschke Santal\'o inequality $|L| |L^\circ| \geq c^n    |B_2^n|^2$ from Theorem 1 of \cite{BM87}, and (\ref{lam=0}).
As $ \int_{\mathbb{R}^n} g^{(s)}_e dx=  \frac{\left|B^{n_s}_2\right|}{s^\frac{n}{2} \left|B^{ \frac{1}{s}}_2\right|}$ and as $\lam+\mu=1$, we get
\begin{eqnarray*}
\frac{OS_\lam^{(s)} (f)}{\left(\int_{\mathbb{R}^n} f dx \right)^{1-2 \lam}}  &\geq& \frac{OS_\mu^{(s)} (g^{(s)}_e)}{\left(\int_{\mathbb{R}^n} g^{(s)}_e dx \right)^{1-2 \lam}} 
  \frac{s^\mu \, C^{n_s \mu}}{ \lfloor {C_0 n_s^{5/6-\mu}}\rfloor} \frac{1}{ L_{f_{(s)}^\circ}^{2n \mu} }\left(\frac{ \left| B_2^{n_{s}}\right|^2}{ n_s^2\,  s^n \, \left|B_2^{\frac{1}{s}}\right|^2}\right)^{1-2\mu} \left( \frac{\left|B^{n_s}_2\right|}{s^\frac{n}{2} \left|B^{ \frac{1}{s}}_2\right|}\right)^{2(\mu-\lam)}\\
 &=& \frac{OS_\mu^{(s)} (g^{(s)}_e)}{\left(\int_{\mathbb{R}^n} g^{(s)}_e dx \right)^{1-2 \lam}} 
  \frac{s^{1-\lam} \, C^{n_s(1- \lam)}}{ \lfloor {C_0 n_s^{5 \lam-13/6}}\rfloor} \frac{1}{ L_{f_{(s)}^\circ}^{2n(1-\lam)} }.
\end{eqnarray*}
\vskip 2mm
\noindent
In the even case we use (\ref{loewner2}) to get the estimates with $n_s^{n_s  \frac{1-2 \lam}{2} }$ instead of $n_s^{n_s (1 -2 \lam)}$ in (i) and (ii).
\end{proof}
\vskip 4mm
\noindent
\subsubsection{Proof of Proposition \ref{prelprop1}.}

\begin{proof}
The  lower bounds of (i) and (ii),  together with the equality cases, are Theorem \ref{basic}, (iii).
\vskip 2mm
\noindent
Let $\lam \in (-\infty, 0]$. As in the proof of Theorem \ref{prelprop} (i) we choose again $h$ such that $K_s(h)=L(K_s(f))$. Similar as in that proof,  
we get with  (\ref{loewner}), (\ref{f-volumen})  and $\frac{n_s-p}{n_s+p} = 1-2 \lam$,
\begin{eqnarray*}
os_\lam^{(s)} (f)  &\leq& as_{\lam} ^{(s)}(h) = \frac{as_{p} \left(L(K_s(f) )\right)}{ n_s  s^\frac{n}{2} \, \left|B_2^{\frac{1}{s}}\right| } 
= \left(\frac {\left|L(K_s(f))\right|}{\left|B^{n_s}_2\right|} \right) ^\frac{n_s-p}{n_s+p}  \frac{ |B^{n_s}_2|}{ s^\frac{n}{2} \,  \left|B_2^{\frac{1}{s}}\right|} \\
&\leq&  n_s^{n_s (1-2 \lam)} \left(\frac {\left|K_s(f)\right|}{s^\frac{n}{2} \,  \left|B_2^{\frac{1}{s}}\right|} \right) ^{1-2 \lam} \frac{ |B^{n_s}_2|^{2 \lam}}{ \left(s^\frac{n}{2} \, \left|B_2^{\frac{1}{s}}\right|\right)^{2 \lam}} \\
&=&  n_s^{n_s (1-2 \lam)} \left(\int_{\mathbb{R}^n} f\, dx \right) ^{1-2 \lam} \frac{os_\lam^{(s)} (g^{(s)}_e) }{\left(\int_{\mathbb{R}^n}  g^{(s)}_e\, dx \right) ^{1-2 \lam}}.
\end{eqnarray*}
\vskip 2mm
\noindent
(ii) The proof of (ii) follows along the same lines as proof of the  estimate  (\ref{below2}) of Theorem \ref{prelprop} (ii).
\vskip 2mm
\noindent
In the even case we again use (\ref{loewner2}) to get the estimates with $n_s^{n_s  \frac{1-2 \lam}{2} }$ instead of $n_s^{n_s (1 -2 \lam)}$.
\end{proof}
\vskip 4mm
\noindent
\subsection{Proof of Theorem \ref{estimate} and Proposition \ref{estimate-prop}.}

We need more notations. We put
\begin{equation}\label{hyper}
G(f)=\{ (x,y) \in \mathbb{R}^n \times \mathbb{R}: x \in \overline{S_f}, -f^s(x) \leq y \leq f^s(x)\}.
\end{equation}
Then $G(f)$ is a convex body in $\mathbb{R}^{n+1}$ and 
\begin{equation}\label{Gf=iso}
|G(f)| = 2 \int_{\mathbb{R}^n} f^s dx.
\end{equation}

\vskip 4mm
\subsubsection{Proof of Theorem \ref{estimate}.}
\par
\noindent
The upper estimates  and the equality cases of (i) and (ii)  follow again from Theorem \ref{basic}, (iii).
\vskip 2mm
\noindent
{\bf The lower estimates.}
\par
\noindent
(i) Let  $s \geq 1$.
We will assume that $G(f)$ is in L\"owner position  i.e.,  the L\"owner ellipsoid $L(G(f))$
is the Euclidean ball $\left(\frac{|L(G(f))|}{|B^{n+1}_2|}\right)^\frac{1}{(n+1)} B^{n+1}_2 = R \, B^{n+1}_2$. Moreover, 
\begin{equation}\label{loewner3}
G(f) \subset L(G(f)) \subset (n+1) \ G(f),
\end{equation}
and for a $0$-symmetric convex body $G(f)$, which happens if $f$ is even,
\begin{equation}\label{loewner4}
G(f) \subset L(G(f)) \subset \sqrt{n+1}  \ G(f).
\end{equation} 
We choose $h=h_R$ such that $G(h)= L(G(f)) = R\, B^{n+1}_2$. Then $h \geq f$, $h \in C_s^+(\mathbb{R}^n)$ and 
$$
h(x) =
h_R= \left(R^2 - |x|^2\right)^\frac{1}{2 s}= s^{-\frac{1}{2s}}\left(\left(s^\frac{1}{2} R \right)^2- s |x|^2\right)^\frac{1}{2s}= \frac{g^{(s)}_{e, s^\frac{1}{2} R}}{s^{\frac{1}{2s}}}.
$$
As for $\alpha >0$, $\frac{OS_\lam^{(s)} (\alpha f) }{\left(\int_{\mathbb{R}^n} \alpha  f \, dx \right) ^{1-2 \lam}} =\frac{OS_\lam^{(s)} (f) }{\left(\int_{\mathbb{R}^n}  f \, dx \right) ^{1-2 \lam}} $
 we can assume that  $0 \leq f(x) \leq 1$. If not, replace $f$ by $\frac{f}{\|f\|_\infty}$. Then for $s \geq 1$, $\int_{\mathbb{R}^n}  f  \geq \int_{\mathbb{R}^n}  f^s $ and 
thus,  with   (\ref{loewner3}), Lemma \ref{lemma-phie},  the definition of $R$,  and as $\lam \in [1/2, 1]$, 
\begin{eqnarray}\label{OS-s>1}
\frac{OS_\lam^{(s)} (f) }{\left(\int_{\mathbb{R}^n}  f \, dx \right) ^{1-2 \lam}}  & \geq& \frac{OS_\lam^{(s)} (f) }{\left(\int_{\mathbb{R}^n}  f^s \, dx \right) ^{1-2 \lam}} 
\geq  \frac{as_{\lam} ^{(s)}(h_R)}{2^{2 \lam-1} |G(f)|^{1-2 \lam}} \nonumber \\
&=& \frac{R^{(n+\frac{1}{s}) (1- 2 \lam)} }{2^{2 \lam-1} |G(f)|^{1-2 \lam}}
\frac{ \left(\int_{\mathbb{R}^n} g^{(s)}_e\right) ^{1-2\lam}  }{s^{\frac{n}{2}( 2 \lam -1)} }   \, \frac{OS_\lam^{(s)} (g^{(s)}_e) }{\left(\int_{\mathbb{R}^n}g^{(s)}_e\right) ^{1-2\lam} } \nonumber \\
&\geq& \frac{OS_\lam^{(s)} (g^{(s)}_e) }{\left(\int_{\mathbb{R}^n} g^{(s)}_e\right) ^{1-2\lam} } \frac{c^{n(2 \lam -1)}}{(n+1)^{(2\lam -1)(n+1)}} 
 \left(\frac{\Gamma\left(\frac{n+1/s}{2}+1\right)}{\Gamma\left(\frac{n+1}{2}+1\right)\Gamma\left(\frac{1}{2s}+1\right)} \right)^{2\lam-1} \nonumber\\
&\geq&\frac{c^ {n(2 \lam-1)} }{(n+1)^{(2\lam -1)(3n/2+1)}}
\frac{OS_\lam^{(s)} (g^{(s)}_e) }{\left(\int _{\mathbb{R}^n}g^{(s)}_e\right) ^{1-2\lam} }.
\end{eqnarray}
In the last inequality we have used that $\frac{\Gamma\left(\frac{n+1/s}{2}+1\right)}{\Gamma\left(\frac{n+1}{2}+1\right)} \geq c/ n^\frac{1}{2}$
and that $\Gamma\left(\frac{1}{2s}+1\right) ^{-1}\geq 1$. 
\vskip 2mm
\noindent
(ii) The estimate for $\frac{IS_\lam^{(s)} (f) }{\left(\int_{\mathbb{R}^n}  f \, dx \right) ^{1-2 \lam}}$ follows as usual from the one in (i) by duality.
\vskip 2mm
\noindent
In the even case (\ref{loewner4}) gives estimates with $(n+1)^{(n+1) \frac{1-2 \lam}{2}}$ instead of $(n+1)^{(n+1) (1 -2 \lam)}$.
\vskip 3mm
\noindent
\subsubsection{Proof of Proposition \ref{estimate-prop}.}

\begin{proof} 
Let $s \geq 1$.
The  lower bounds of (i) and (ii),  together with the equality cases, are Theorem \ref{basic}, (iii).
\vskip 2mm
\noindent
(i) Let $\lam \in (- \infty, 0]$. We apply the same steps  as for estimate  (\ref{OS-s>1}). We put again $G(f)$ in L\"owner position,   i.e.,  the L\"owner ellipsoid $L(G(f))$
is the Euclidean ball $\left(\frac{|L(G(f))|}{|B^{n+1}_2|}\right)^\frac{1}{(n+1)} B^{n+1}_2 = R \, B^{n+1}_2$.
In the same way as for  (\ref{OS-s>1}), we get with an absolute constant $c>0$, 
\begin{eqnarray*}
\frac{os_\lam^{(s)} (f) }{\left(\int_{\mathbb{R}^n}  f \, dx \right) ^{1-2 \lam}}  & \leq&
c^{n(1-2 \lam)}\, (n+1)^{(3n/2+1)(1-2 \lam)}
\frac{os_\lam^{(s)} (g^{(s)}_e) }{\left(\int_{\mathbb{R}^n} g^{(s)}_e\right) ^{1-2\lam} }.
\end{eqnarray*}
\vskip 2mm
\noindent
(ii) follows from (i) by duality. 
\vskip 2mm
\noindent
In the even case (\ref{loewner4}) gives estimates with $(n+1)^{(3n/2+1) \frac{1-2 \lam}{2}}$ instead of $(n+1)^{(3n/2+1)(1-2\lambda)}$.

\end{proof}

\vskip 4mm
\noindent

\subsection{Appendix: Proof of Lemma \ref{lemLphi}.}
\vskip 2mm
\noindent
\begin{proof}
The statement (ii) follows immediately from Lemma  3.1 in \cite{ArtKlarMil} which says that $K_s(f)^\circ=K_s(f ^\circ_{(s)})$. Therefore it is enough to prove statement (i).
\par
\noindent
\begin{eqnarray*}
L_{K_s(f)}&=&L_{1_{K_{s(f)}}}=
\left(\frac{\sup_{z \in \mathbb{R}^{n+\frac{1}{s}}} 1_{K_{s(f)}}(z)}{\int_{\mathbb{R}^{n+\frac{1}{s}}} 1_{K_{s(f)}} (z)dz} \right)^\frac{1}{n+\frac{1}{s}}\, \left( \det \, Cov(1_{K_{s(f)}})\right)^\frac{1}{2\left(n+\frac{1}{s}\right)}\\
&=&\frac{1}{\left(\text{vol}_{n+\frac{1}{s}}(K_s(f))\right)^\frac{1}{n+\frac{1}{s}}}\, \left( \det \, Cov(1_{K_{s(f)}})\right)^\frac{1}{2\left(n+\frac{1}{s}\right)}
\end{eqnarray*}
We compute the determinant of the matrix $\left[ Cov(1_{K_{s(f)}})\right]_{ i,j }$. 
We use that the $s$-concave function  $f$ has center of gravity at $0$ and consequently $K_{s(f)}$ has center of gravity at $0$. Thus
\begin{eqnarray*}
\det \, \left(\left[ Cov(1_{K_{s(f)}})\right]_{ i,j }\right)&=& \det \, \left( \left[\frac{\int_{K_{s(f)}} z_i z_j \,  dz}{\int_{K_{s(f)}}  dz} - 
\frac{\int_{K_{s(f)}} z_i \,  dz}{\int_{K_{s(f)}}dz} \, \frac{\int_{K_{s(f)}} z_j\, dz}{\int_{K_{s(f)}} dz} \right]_{i,j}\right)\\
&=&  \det \, \left( \left[ \frac{\int_{K_{s(f)}} z_i z_j \,  dz}{\int_{K_{s(f)}}  dz} \right]_{i,j}\right) = 
\frac{\det \, \left( \left[ \int_{K_{s(f)}} z_i z_j \,  dz \right]_{i,j}\right)}{\left(\mbox{vol}_{n+\frac{1}{s}}\left(K_s(f)\right)\right)^{n+\frac{1}{s}} } .
\end{eqnarray*}
Let $$z=\left(x_1, \cdots, x_n, y_{n+1}, \cdots, y_{n+\frac{1}{s}}\right) \in \mathbb{R}^n \times \mathbb{R}^\frac{1}{s}.$$
We have to  compute the determinant of the following $[n+\frac{1}{s}]\times [n+ \frac{1}{s}]$ matrix
\begin{align*}
&\left[
\begin{array}{c;{2pt/2pt}c}
\begin{matrix}
\int_{K_s(f)} x_1^2 dz & \cdots & \cdots & \int_{K_s(f)} x_1x_n dz\\
\int_{K_s(f)} x_2x_1 dz & \cdots & \cdots & \int_{K_s(f)} x_2x_n dz\\
\vdots & \vdots & \ddots & \vdots\\
\int_{K_s(\varphi)} x_nx_1 dz & \cdots & \cdots & \int_{K_s(f)} x_n^2 dz\\
\end{matrix} & \begin{matrix}
\int_{K_s(f)} x_1y_{n+1} dz & \cdots & \int_{K_s(f)} x_1y_{n+\frac{1}{s}} dz\\
\int_{K_s(f)} x_2y_{n+1} dz & \cdots & \int_{K_s(f)} x_2y_{n+\frac{1}{s}} dz\\
\vdots & \ddots & \vdots\\
\int_{K_s(f)} x_ny_{n+1} dz & \cdots & \int_{K_s(f)} x_ny_{n+\frac{1}{s}} dz\\
\end{matrix}\\\hdashline[2pt/2pt]
\begin{matrix}
\int_{K_s(f)} y_{n+1}x_1 dz & \cdots & \cdots & \int_{K_s(f)} y_{n+1}x_n dz\\
\vdots & \vdots & \ddots & \vdots\\
\int_{K_s(f)} y_{n+\frac{1}{s}}x_1 dz & \cdots & \cdots & \int_{K_s(f)} y_{n+\frac{1}{s}}x_n dz\\
\end{matrix} & \begin{matrix}
\int_{K_s(f)} y_{n+1}^2 dz & \cdots & \int_{K_s(f)} y_{n+1}y_{n+\frac{1}{s}} dz\\
\vdots & \ddots & \vdots\\
\int_{K_s(f)} y_{n+\frac{1}{s}}y_{n+1} dz & \cdots & \int_{K_s(f)}y_{n+\frac{1}{s}}^2 dz\\\
\end{matrix}
\end{array}
\right]
\end{align*}
For all $n+1 \leq i \leq n+\frac{1}{s}$, 
$$
\int_{B^\frac{1}{s} \left( \frac{x}{\sqrt{s}}, f ^s \left( \frac{x}{\sqrt{s}}\right) \right)} y_i dy=0.
$$
Moreover, for all $1 \leq i,j \leq \frac{1}{s}$, $i \neq j$, 
$$\int_{K_s(f)}y_{n+i}y_{n+j}dz = 0.$$
Hence the above matrix reduces to 
$$\left[
\begin{array}{c;{2pt/2pt}c}
\begin{matrix}
\int_{K_s(f)} x_1^2 dz & \int_{K_s(f)} x_1x_2 dz & \cdots & \int_{K_s(f)} x_1x_n dz\\
\int_{K_s(f)} x_2x_1 dz & \int_{K_s(f)} x_2^2 dz & \cdots & \int_{K_s(f)} x_2x_n dz\\
\vdots & \vdots & \ddots & \vdots\\
\int_{K_s(f)} x_nx_1 dz & \int_{K_s(f)} x_nx_2 dz & \cdots & \int_{K_s(f)} x_n^2 dz\\
\end{matrix} & \begin{matrix}
\quad 0\quad & \quad \cdots \quad & \quad 0 \quad\\
\quad 0 \quad & \quad \cdots \quad & \quad 0 \quad\\
\vdots & \ddots & \vdots\\
\quad 0 \quad & \quad \cdots \quad  & \quad 0\quad \\
\end{matrix}\\\hdashline[2pt/2pt]
\begin{matrix}
\quad \quad 0\quad  & \quad \quad \quad \quad 0 \quad  & \quad \quad \text{ }\cdots  \quad & \quad\quad 0 \quad \quad\\
\quad\quad \vdots \quad  & \quad \quad \quad \quad \vdots  \quad & \quad \quad \text{ }\ddots \quad & \quad \quad \vdots \quad \quad\\
\quad \quad 0 \quad  & \quad \quad \quad \quad 0 \quad  & \quad \quad \text{ }\cdots \quad   & \quad \quad 0\quad \quad \\
\end{matrix} & \begin{matrix}
\int_{K_s(f)} y_{n+1}^2 dz &  & \\
 & \ddots & \\
 &  & \int_{K_s(f)}y_{n+\frac{1}{s}}^2 dz\\
\end{matrix}
\end{array}
\right]$$
For all $1 \leq i,j \leq n$, 
\begin{eqnarray*}
 \int_{K_{s(f)}} x_i x_j  dz = \int_{S_f} x_i x_j  \int_{B^\frac{1}{s} \left( \frac{x}{\sqrt{s}}, f^s \left( \frac{x}{\sqrt{s}}\right) \right)} dy dx
= s^\frac{n+2}{2} \, \mbox{vol}_{\frac{1}{s}} \left(B_2^{\frac{1}{s}}\right)  \int_{S_f} x_i x_j  f(x) dx.
\end{eqnarray*}
For all $n+1 \leq i,  j  \leq n+\frac{1}{s}$
\begin{eqnarray*}
 \int_{K_{s(f)}} y_i^2 dz =  \int_{K_{s(f)}} y_j^2 dz
\end{eqnarray*}
Thus the determinant of the above matrix is  composed of an $n \times n$ matrix and a $\frac{1}{s} \times \frac{1}{s}$ diagonal matrix.
The determinant of the $n \times n$ matrix is
\begin{align*}
&\det \left( s^{\frac{n+2}{2}} \left|B_2^\frac{1}{s}\right| \int_{S_f} x_ix_j f(x)dx \right)_{1 \leq i, j \leq n}
 =s^{\frac{n(n+2)}{2}} \left|B_2^\frac{1}{s}\right|^n \left( \int_{\mathbb{R}^n} f dx \right)^n \det (Cov(f))_{i,j}.
\end{align*}
Using polar coordinates and several changes of variable, the determinant of the $\frac{1}{s} \times \frac{1}{s}$ diagonal matrix becomes
\begin{equation*}
\left( \int_{K_s(f)} y_{n+\frac{1}{s}}^2 dz \right)^\frac{1}{s} = \left(\int_{S_f}\int_{B_2^\frac{1}{s}\left(x, f^s\left(\frac{x}{\sqrt{s}}\right)\right)} y_{n+\frac{1}{s}}^2 dy \, dx \right)^\frac{1}{s}\\
=\left(\frac{ \left|B_2^\frac{1}{s}\right|s^\frac{n}{2}}{\left( \frac{1}{s}+2 \right)} \right)^\frac{1}{s}\left( \int_{\mathbb{R}^n} f^{2s+1} dx \right)^\frac{1}{s}.
\end{equation*}
Thus, we have
\begin{align*}
&(\det \, Cov(1_{K_s(f)}))^{\frac{1}{2(n+\frac{1}{s})}}
=\left(\frac{\det\left(\left[ \int_{K_s(f)} z_iz_j dz\right]_{i,j}\right)}{\left(\text{vol}_{n+\frac{1}{s}}(K_s(f))\right)^{n+\frac{1}{s}}}\right)^{\frac{1}{2\left(n+\frac{1}{s}\right)}}\\
&= \frac{\left(\left(\int_{K_s(f)}y_{n+\frac{1}{s}}^2 dz\right)^\frac{1}{s} \det\left(\left[ \int_{K_s(f)}x_ix_j dz \right]_{1\leq i,j \leq n}\right)\right)^{\frac{1}{2\left(n+\frac{1}{s}\right)}}}{\left(\text{vol}_{n+\frac{1}{s}}(K_s(f))\right)^\frac{1}{2}}\\
&= \frac{\left(\text{vol}_\frac{1}{s} \left(B_2^\frac{1}{s}\right)^{n+\frac{1}{s}}s^{\frac{n}{2s}+\frac{n(n+2)}{2}}\left( \int_{\mathbb{R}^n} f^{2s+1} dx \right)^\frac{1}{s} \left( \int_{\mathbb{R}^n} f dx \right)^n \det (Cov(f))_{i,j}\right)^\frac{1}{2\left( n+\frac{1}{s} \right)}}{\left(\frac{1}{s}+2\right)^{\frac{1}{2s\left(n+\frac{1}{s}\right)}}\left(\text{vol}_{n+\frac{1}{s}}(K_s(f))\right)^{\frac{1}{2}}}.
\end{align*}
Finally, also using (\ref{vol-Ks}),  
\begin{align*}
&L_{K_s(f)} = \frac{\left( \det \, Cov(1_{K_{s(f)}})\right)^\frac{1}{2\left(n+\frac{1}{s}\right)}}{\left(\text{vol}_{n+\frac{1}{s}} K_s(f)\right)^\frac{1}{n+\frac{1}{s}}}\\
&= \frac{\left(\text{vol}_\frac{1}{s} \left(B_2^\frac{1}{s}\right)^{n+\frac{1}{s}}s^{\frac{n}{2s}+\frac{n(n+2)}{2}}\left( \int_{\mathbb{R}^n} f^{2s+1} dx \right)^\frac{1}{s} \left( \int_{\mathbb{R}^n} f dx \right)^n \det (Cov(f))_{i,j}\right)^\frac{1}{2\left( n+\frac{1}{s} \right)}}{\left(\frac{1}{s}+2\right)^{\frac{1}{2s\left(n+\frac{1}{s}\right)}}\left(\text{vol}_{n+\frac{1}{s}}(K_s(f))\right)^{\frac{1}{n+\frac{1}{s}}+\frac{1}{2}}}\\
&= \left(\frac{\left( \int_{\mathbb{R}^n} f^{2s+1} dx \right)^\frac{1}{2s} (\det (Cov(f))_{i,j}^\frac{1}{2}}{2^\frac{1}{2s}\left(1+\frac{1}{2s}\right)^{\frac{1}{2s}}\text{vol}_\frac{1}{s} \left(B_2^\frac{1}{s}\right)\left( \int_{\mathbb{R}^n} f dx \right)^\frac{1}{2s}\left( \int_{\mathbb{R}^n} f dx \right)}\right)^\frac{1}{n+\frac{1}{s}}\\
&= L_f^\frac{n}{n+\frac{1}{s}} \cdot\left(\frac{\left( \int_{\mathbb{R}^n} f^{2s+1} dx \right)^\frac{1}{2s}}{2^\frac{1}{2s}\left(1+\frac{1}{2s}\right)^{\frac{1}{2s}} \|f\|_\infty \left( \int_{\mathbb{R}^n} f dx \right)^\frac{1}{2s}\text{vol}_\frac{1}{s}\left(B_2^\frac{1}{s}\right)}\right)^\frac{1}{n+\frac{1}{s}}
\end{align*}
\end{proof}

\par
\noindent 
Stephanie Egler\\
{\small Department of Mathematics \\
{\small Case Western Reserve University \\
{\small Cleveland, Ohio 44106, U. S. A. \\
{\small \tt sae58@case.edu}\\ \\
\noindent
\and 
Elisabeth Werner\\
{\small Department of Mathematics \ \ \ \ \ \ \ \ \ \ \ \ \ \ \ \ \ \ \ Universit\'{e} de Lille 1}\\
{\small Case Western Reserve University \ \ \ \ \ \ \ \ \ \ \ \ \ UFR de Math\'{e}matique }\\
{\small Cleveland, Ohio 44106, U. S. A. \ \ \ \ \ \ \ \ \ \ \ \ \ \ \ 59655 Villeneuve d'Ascq, France}\\
{\small \tt elisabeth.werner@case.edu}\\ \\

\end{document}